\documentclass[reqno]{amsart}
\usepackage[T1]{fontenc}
\usepackage{dsfont}
\usepackage{mathrsfs}
\usepackage[colorlinks, citecolor=blue, linkcolor=blue]{hyperref}
\usepackage{xcolor}
\usepackage[a4paper,asymmetric]{geometry}
\usepackage{mathscinet}
\usepackage{latexsym}
\usepackage{amsthm}
\usepackage{amssymb}
\usepackage{amsfonts}
\usepackage{amsmath}
\usepackage{longtable}
\usepackage{graphicx}
\usepackage{multirow}
\usepackage{multicol}
\usepackage{amsfonts, amsmath}
\usepackage{latexsym,bm,amsfonts,amssymb,pifont,mathbbol,bbm}
\usepackage{verbatim}

\newtheorem{theorem}{Theorem}[section]
\newtheorem{thm}[theorem]{Theorem}
\newtheorem{lemma}[theorem]{Lemma}
\newtheorem{lem}[theorem]{Lemma}
\newtheorem{remark}[theorem]{Remark}
\newtheorem{proposition}[theorem]{Proposition}
\newtheorem{prop}[theorem]{Proposition}
\newtheorem{corollary}[theorem]{Corollary}
\newtheorem{hyp}[theorem]{HYPOTHESIS}
\theoremstyle{definition}

\newtheorem{defn}[theorem]{Definition}
\newtheorem{ex}[theorem]{Example}
 \newtheorem{nott}{Notation}
\theoremstyle{remark}
\numberwithin{equation}{section}

 \DeclareMathAlphabet{\mathpzc}{OT1}{pzc}{m}{it}
 \DeclareMathAlphabet{\mathsfsl}{OT1}{cmss}{m}{sl}

  \newcommand{\FH}{\mathfrak{H}}

\newcommand{\dif}{\mathrm{d}}

\newcommand{\abs}[1]{\left\vert#1\right\vert}
\newcommand{\set}[1]{\left\{#1\right\}}

\newcommand{\norm}[1]{\left\Vert#1\right\Vert}

\newcommand{\E}{\mathbb{E}}

 \newcommand{\tensor}[1]{\mathsf{#1}}

 \newcommand{\Rnum}{\mathbb{R}}

 \newcommand{\innp}[1]{\langle {#1}\rangle}
\newcommand{\Be}{\begin{equation}}
\newcommand{\Ee}{\end{equation}}
\newcommand{\Bs}{\begin{split}}
\newcommand{\Es}{\end{split}}
\newcommand{\Bes}{\begin{equation*}}
\newcommand{\Ees}{\end{equation*}}
\newcommand{\BT}{\begin{thm}}
\newcommand{\ET}{\end{thm}}
\newcommand{\Bp}{\begin{proof}}
\newcommand{\Ep}{\end{proof}}
\newcommand{\BL}{\begin{lem}}
\newcommand{\EL}{\end{lem}}
\newcommand{\BP}{\begin{proposition}}
\newcommand{\EP}{\end{proposition}}
\newcommand{\BC}{\begin{corollary}}
\newcommand{\EC}{\end{corollary}}
\newcommand{\BR}{\begin{remark}}
\newcommand{\ER}{\end{remark}}
\newcommand{\BD}{\begin{defn}}
\newcommand{\ED}{\end{defn}}
\newcommand{\BI}{\begin{itemize}}
\newcommand{\EI}{\end{itemize}}

\newcommand{\red}{\textcolor[rgb]{1.00,0.00,0.00}}

\allowdisplaybreaks

\begin{document}
\title[Parameter estimation for Ornstein-Uhlenbeck processes]{Parameter estimation for an Ornstein-Uhlenbeck Process driven by a general Gaussian noise}
\author[Y. Chen]{Yong CHEN}
 \address{College of Mathematics and Information Science, Jiangxi Normal University, Nanchang, 330022, Jiangxi, China}
\email{zhishi@pku.org.cn;\, chenyong77@gmail.com}
 \author[H.Zhou]{Hongjuan ZHOU}
 \address{ School of Mathematical and Statistical Sciences, Arizona State University,\,Arizona, USA. }
 \email {Hongjuan.Zhou@asu.edu}
\begin{abstract}

In this paper, we consider an inference problem for an Ornstein-Uhlenbeck process driven by a general one-dimensional centered Gaussian process $(G_t)_{t\ge 0}$. The second order mixed partial derivative of the covariance function $ R(t,\, s)=\mathbb{E}[G_t G_s]$ can be decomposed into two parts, one of which coincides with that of fractional Brownian motion and the other is bounded by $(ts)^{\beta-1}$ up to a constant factor.
This condition is valid for a class of continuous Gaussian processes that fails to be self-similar or have stationary increments. Some examples include the subfractional Brownian motion and the bi-fractional Brownian motion. Under this assumption, we study the parameter estimation for drift parameter in the Ornstein-Uhlenbeck process driven by the Gaussian noise $(G_t)_{t\ge 0}$. For the least squares estimator and the second moment estimator constructed from the continuous observations, we prove the strong consistency and the asympotic normality, and obtain the Berry-Ess\'{e}en bounds.  The proof is based on the inner product's representation of the Hilbert space $\mathfrak{H}$ associated with the Gaussian noise $(G_t)_{t\ge 0}$, and the estimation of the inner product based on the results of the Hilbert space associated with the fractional Brownian motion.\\


\end{abstract}
\maketitle

\section{ Introduction}\label{sec 03}

We are interested in the statistical inference for the Ornstein-Uhlenbeck process defined by the following stochastic differential equation (SDE)
\begin{equation}\label{fOU}
\mathrm{d} X_t= -\theta X_t\mathrm{d} t+ \sigma \mathrm{d}G_t,\quad  t \in [0,T], \ T >0
\end{equation} where $X_0=0$ and $(G_t)_{t\ge 0}$ is a general one-dimensional centered Gaussian process. We note that the volatility parameter $\sigma>0$ can be estimated by power variation method (for example, see \cite{il97}, \cite{KM 12}). Without loss of generality, we will assume that $\sigma=1$. Suppose that only one trajectory $(X_t, t \geq 0)$ can be obtained. We would like to construct a consistent estimator for the unknown drift parameter $\theta>0$ and study its asymptotic behavior.

When the Gaussian process is Brownian motion, the statistical inference problem about the parameter $\theta$ has been intensively studied over the past decade (see \cite{k04}, \cite{ls01} and the references therein). In the fractional Brownian motion case, the consistency property for the maximum likelihood estimation (MLE) method was obtained in \cite{KLB 02}, \cite{tv 07}, and the central limit theorem was proved in \cite{bcs11}, \cite{bk10}. The least squares method was studied in \cite{hu Nua 10} and its asymptotic behavior was proved for $H \in (\frac{1}{2}, \frac{3}{4})$. Then in \cite{hu nua zhou 19}, these results were generalized for $H \in (0, 1)$. We would like to mention some work for the non-ergodic case as well, i.e., $\theta<0$. For the Brownian motion case, MLE was studied in \cite{bs83}, \cite{dk03} and the limiting distribution is Cauchy. The least squares estimation in the case of fractional Brownian motion and other Gaussian processes was considered in \cite{beo11}, \cite{eeo16}, \cite{mendy 13} and the references therein. Recently, the MLE in the case of sub-fractional Brownian motion case was investigated in \cite{dmm11}. In this paper, we would like to discuss  the case where $\theta>0$ and the noise is a general Gaussian process $(G_t)_{t\ge 0}$ that fails to be self-similar or have stationary increments. We assume that {the process $G_t$} satisfies the following Hypothesis \ref{hypthe 1}.

\begin{hyp}\label{hypthe 1}
For $\beta\in (\frac12, \,1)$, the covariance function $ R(t,\, s)=\mathbb{E}[G_t G_s]$ for any $t\neq s \in [0,\infty)$ satisfies
\begin{align*}
\frac{\partial^2}{\partial t\partial s}R(t,s)= C_{\beta}\abs{t-s}^{2\beta -2}+\Psi(t,\,s),
\end{align*}with 
\begin{align*}
\abs{\Psi(t,\,s)}\le C'_{\beta} \abs{ts}^{\beta -1},
\end{align*} where the constants $\beta,\,C_{\beta}>0,\,C'_{\beta}\ge 0$ do not depend on $T$.
Moreover, for any $t \geq 0$, $R(0,t)=0$.
\end{hyp}
We will see that sub-fractional Brownian motion, bi-fractional Brownian motion and some other Gaussian processes are special examples to satisfy the Hypothesis \ref{hypthe 1}. Recall that the idea to construct the least squares estimator (LSE) for the drift coefficient $\theta$ is to minimize 
  $$\int_0^T |\dot X_t + \theta X_t|^2 dt$$
(see \cite{hu Nua 10}, \cite{hu nua zhou 19})\,. In this way, we obtain the LSE defined by
\begin{equation}\label{hattheta}
\hat{\theta}_T=-\frac{\int_0^T X_t\mathrm{d}X_t}{\int_0^T X_t^2\mathrm{d} t}=\theta-\frac{\int_0^T X_t\mathrm{d}G_t}{\int_0^T X_t^2\mathrm{d} t},
\end{equation}where the integral with respect to $G$ is interpreted in the Skorohod sense (or say a divergence-type integral).

We will also study the second moment estimator that is given by 
\begin{align}\label{theta tilde formula}
\tilde{\theta}_{T}=\Big( \frac{1}{C_{\beta} \Gamma(2\beta-1) T} \int_0^T X_t^2\mathrm{d} t \Big)^{-\frac{1}{2\beta}}.
\end{align}
In this paper, we will prove the strong consistency and the central limit theorems for the two estimators. The Berry-Ess\'{e}en bounds will be also obtained. These results are stated in the following theorems.

\begin{thm}\label{thm strong}
When Hypothesis~\ref{hypthe 1} is satisfied, both the least squares estimator $\hat{\theta}$ and the  second moment estimator  $\tilde{\theta}_{T}$ are strongly consistent, i.e., 
\begin{align*}
\lim_{T\to\infty}\hat{\theta}_T=\theta,\qquad \lim_{T\to\infty}\tilde{\theta}_T=\theta, \qquad a.s..
\end{align*}
\end{thm}

\begin{thm}\label{main thm 2}
Assume $\beta\in (\frac12,\,\frac34)$ and Hypothesis~\ref{hypthe 1} is satisfied. Then, both $\sqrt{T}( \hat{\theta}_T-\theta )$ and $\sqrt{T}( \tilde{\theta}_T-\theta ) $ are asymptotically normal as $T\to \infty$. Namely,
\begin{align}
\sqrt{T}( \hat{\theta}_T-\theta )& \stackrel{ {law}}{\to}  { \mathcal{N}(0, \,\theta \sigma_{\beta}^2)},\label{asy norm LSE}\\
\sqrt{T}( \tilde{\theta}_T-\theta )&\stackrel{ {law}}{\to}  \mathcal{N}(0, \,\theta \sigma_{\beta}^2 /4\beta^2),\label{asy norm SME}
\end{align}where 
\begin{align*}
\sigma_{\beta}^2= (4\beta-1) \big( 1+\frac{\Gamma(3-4\beta)\Gamma(4\beta-1)}{\Gamma(2\beta)\Gamma(2-2\beta)}\big).
\end{align*}
\end{thm}

\begin{thm}\label{B-E bound thm}
Let $Z$ be a standard Gaussian random variable.
Assume $\beta\in (\frac12,\,\frac34)$ and Hypothesis~\ref{hypthe 1} is satisfied. Then, there exists a constant $C_{\theta, \beta} > 0$ such that when $T$ is large enough, 
\begin{equation}\label{b-e bound 34}
\sup_{z\in \Rnum}\abs{P(\sqrt{\frac{T}{\theta \sigma^2_{\beta}}} (\hat{\theta}_T-\theta )\le z)-P(Z\le z)}\le\frac{ C_{\theta, \beta}}{{T^{\gamma}}},
\end{equation}
and 
\begin{equation}\label{b-e bound 44}
\sup_{z\in \Rnum}\abs{P(\sqrt{\frac{4\beta^2 T}{\theta \sigma^2_{\beta}}} (\tilde{\theta}_T-\theta )\le z)-P(Z\le z)}\le\frac{ C_{\theta, \beta}}{{T^{\frac{3-4\beta}{2}}}},
\end{equation}
where 
\begin{equation*}
\gamma=\left\{
      \begin{array}{ll}
 \frac12, & \quad \text{if }  \beta\in (\frac12,\,\frac58),\\
 {  \frac12 -,} & \quad \text{if } \beta = \frac58, \\
3-4\beta, &\quad \text{if } \beta\in (\frac58,\, \frac34).
  \end{array}
\right.
\end{equation*}
\end{thm}

Next, we give some well known processes that satisfy the Hypothesis ~\ref{hypthe 1}. 
\begin{ex}
 Clearly the fractional Brownian motion $\{B^H(t), t\geq 0\}$ with covariance function
    $$R(s,t)= \frac{1}{2}(|s|^{2H} + |t|^{2H} - |t-s|^{2H}),$$
satisfies Hypothesis~\ref{hypthe 1} when $\beta := H>\frac12$. In this case, the upper Berry-Ess\'{e}en bound (\ref{b-e bound 44}) can be improved to  $\frac{ C }{{T^{\gamma}}}$ from the proof of Theorem~\ref{B-E bound thm} and Remark~\ref{rem 36} in this paper. This improved upper Berry-Ess\'{e}en bound is sharper than the one given by Proposition 4.1 (ii) of \cite{SV 18}.
\end{ex}
\begin{ex}
The subfractional Brownian motion $\{S^H(t), t \geq 0\}$ with parameter $H\in (0,1)$ has the covariance function
$$R(t,s)=s^{2H}+t^{2H}-\frac{1}{2}\left((s+t)^{2H}+|t-s|^{2H}\right),$$
which satisfies Hypothesis~\ref{hypthe 1} when $\beta := H>\frac12$. 
\end{ex}
This answers the unsolved problem in \cite{cai xiao 18} where the strong  consistency is unknown for sub-fractional Brownian motion.
\begin{ex}
The bi-fractional Brownian motion $\{B^{H,K}(t), t\geq 0\}$ with parameters $H,K \in (0, 1)$ has the covariance function
$$R(t,s)=\frac{1}{2^K}\left((s^{2H}+t^{2H})^K - |t-s|^{2HK}\right) ,$$
which satisfies Hypothesis~\ref{hypthe 1} when $\beta := HK>\frac12$. \end{ex}

\begin{ex}\label{exmp5}
The generalized sub-fractional Brownian motion $S^{H,K}(t) $ with parameters $H \in (0, 1),\,K \in[1,2)$ and $HK\in (0,1)$ satisfies Hypothesis~\ref{hypthe 1} when $\beta := HK>\frac12$. The covariance function 
is 
$$ R(t,\, s)= (s^{2H}+t^{2H})^{K}-\frac12 \big[(t+s)^{2HK} + \abs{t-s}^{2HK} \big]$$
(see \cite{SghA 13}). 
\end{ex}

\begin{remark}
If the SDE is driven by a linear combination of independent centered Gaussian processes, the results are still valid as long as each Gaussian process satisfies Hypothesis~\ref{hypthe 1}. In this case, the mixed Gaussian process fails to be self-similar.
\end{remark}

\section{Preliminary}
Denote $G = \set{G_t, t\in [0,T ]}$ as a continuous centered Gaussian process with covariance function
    $$ \E(G_tG_s)=R(s,t), \ s, t \in [0,T], $$
defined on a complete probability space $(\Omega, \mathcal{F}, P)$. The filtration $\mathcal{F}$ is generated by the Gaussian family $G$. Suppose in addition that the covariance function $R$ is continuous.
Let $\mathcal{E}$ denote the space of all real valued step functions on $[0,T]$. The Hilbert space $\mathfrak{H}$ is defined
as the closure of $\mathcal{E}$ endowed with the inner product
\begin{align*}
\innp{\mathbbm{1}_{[a,b)},\,\mathbbm{1}_{[c,d)}}_{\FH}=\E\big(( G_b-G_a) ( G_d-G_c) \big).
\end{align*}
We denote $G=\{G(h), h \in \mathfrak{H}\}$ as the isonormal Gaussian process on the probability space $(\Omega, \mathcal{F}, P)$, indexed by the elements in the Hilbert space $\mathfrak{H}$. In other words, $G$ is a Gaussian family of random variables such that 
   $$\mathbb{E}(G) = \mathbb{E}(G(h)) = 0, \quad \mathbb{E}(G(g)G(h)) = \langle g, h \rangle_{\mathfrak{H}} \,,$$
for any $g, h \in \mathfrak{H}$. 

The following proposition is an extension of Theorem 2.3 of \cite{Jolis 07},  which gives the inner product's  representation of the Hilbert space $\FH$.
\begin{prop}\label{main prelim}
Denote $\mathcal{V}_{[0,T]}$ as the set of bounded variation functions on $[0,T]$. Then $\mathcal{V}_{[0,T]}$ is dense in $\FH$  and we have
\begin{align} \label{innp fg30}
\innp{f,g}_{\FH}=\int_{[0,T]^2} R(t,s) \nu_f( \dif t) \nu_{g}( \dif s),\qquad \forall f,\, g\in \mathcal{V}_{[0,T]}, 
\end{align}
where $\nu_{g}$ is the Lebesgue-Stieljes signed measure associated with $g^0$ defined as
\begin{equation*}
g^0(x)=\left\{
      \begin{array}{ll}
 g(x), & \quad \text{if } x\in [0,T] ;\\
0, &\quad \text{otherwise }.     
 \end{array}
\right.
\end{equation*}
Furthermore, if the covariance function $R(t,s)$ satisfies Hypothesis~\ref{hypthe 1}, then
\begin{align}\label{innp fg3}
\innp{f,g}_{\FH}= \int_{[0,T]^2}  f(t) g(s) \frac{\partial^2 R(t,s)}{\partial t \partial s} \dif t  \dif s, \qquad \forall f,\, g\in \mathcal{V}_{[0,T]}.
\end{align}
\end{prop}
\begin{proof}
The first claim and the identity (\ref{innp fg30}) are rephrased from Theorem 2.3 of \cite{Jolis 07}. The identity (\ref{innp fg3}) can be shown by the routine approximation. 

In fact, Hypothesis~\ref{hypthe 1} implies that 
\begin{align}\label{innp fg301}
\innp{f,g}_{\FH}= \int_{[0,T]^2}  f(t) g(s) \frac{\partial^2R(t,s)}{\partial t \partial s} \dif t  \dif s, \qquad \forall f,\, g\in  \mathcal{E}.
\end{align}
Next, given $f\in \mathcal{V}_{[0,T]}$ and a sequence of partitions $\pi_n= \set{0 = t_0^n < t_1^n < \cdots < t_{k_n}^n = T } $ such that $\pi_n\subset \pi_{n+1}$ and $\abs{\pi_n}\to 0$ as $n \to \infty$, we consider 
\begin{align*}
f_n=\sum_{j=0}^{{k_n}-1} f(t_j^n) \mathbbm{1}_{[t_j^n,t_{j+1}^n)}\,\, \in \mathcal{E}.
\end{align*} Then (A3) and (A4) of \cite{Jolis 07} imply that 
\begin{align*}
\innp{f,f}_{\FH}&=\lim_{n\to \infty}\innp{f_n,f_n}_{\FH}\\
&= \lim_{n\to \infty}\int_{[0,T]^2}  f_n(t) f_n(s) \frac{\partial^2 R(t,s)}{\partial t \partial s} \dif t  \dif s\\
&=\int_{[0,T]^2}  f(t) f(s) \frac{\partial^2 R(t,s)}{\partial t \partial s} \dif t  \dif s,
\end{align*}where the last equality is by Lebesgue's dominated convergence theorem.
Finally, using the polarization identity, we obtain the desired (\ref{innp fg3}).
\end{proof}


\begin{remark} We define the space of measurable functions by 
\begin{align*}
\abs{\mathfrak{H}}=\set{f:[0,\,T]\to \Rnum, \int_0^{T}\int_0^{T} \abs{f(t)f(s)}\frac{\partial^2 R( t,s)}{ \partial t\partial s } \dif t\dif s<\infty}.
\end{align*}
If Hypothesis~\ref{hypthe 1} is satisfied, we understand that the space $\abs{\mathfrak{H}}$ equipped with the inner product 
  $$\langle f, g \rangle_{\mathfrak{H}} = \int_{[0,T]^2} f(t) g(s)\frac{\partial^2 R(t,s)}{\partial t \partial s} \dif t \dif s,  \qquad \forall f, g \in \abs{\mathfrak{H}}$$
is not complete and it is isometric to a proper subspace of $\mathfrak{H}$ (see the fractional Brownian motion case in \cite{Nua} and the references therein). However, Proposition~\ref{main prelim} is good enough to prove the main results of this paper.
\end{remark}


Denote $\mathfrak{H}^{\otimes p}$ and $\mathfrak{H}^{\odot p}$ as the $p$th tensor product and the $p$th symmetric tensor product of the Hilbert space $\mathfrak{H}$. Let $\mathcal{H}_p$ be the $p$th Wiener chaos with respect to $G$. It is defined as the closed linear subspace of $L^2(\Omega)$ generated by the random variables $\{H_p(G(h)): h \in \mathfrak{H}, \ \|h\|_{\mathfrak{H}} = 1\}$, where $H_p$ is the $p$th Hermite polynomial defined by
$$H_p(x)=\frac{(-1)^p}{p!} e^{\frac{x^2}{2}} \frac{d^p}{dx^p} e^{-\frac{x^2}{2}}, \quad p \geq 1,$$
and $H_0(x)=1$. We have the identity $I_p(h^{\otimes p})=H_p(G(h))$ for any $h \in \mathfrak{H}$ where $I_p(\cdot)$ is the generalized Wiener-It\^o stochastic integral. Then the map $I_p$ provides a linear isometry between $\mathfrak{H}^{\odot p}$ (equipped with the norm $\frac{1}{\sqrt{p!}}\|\cdot\|_{\mathfrak{H}^{\otimes p}}$) and $\mathcal{H}_p$. Here $\mathcal{H}_0 = \mathbb{R}$ and $I_0(x)=x$ by convention.

We choose $\{e_k, k \geq 1\}$ to be a complete orthonormal system in the Hilbert space $\mathfrak{H}$. Given $f \in \mathfrak{H}^{\odot m}, g \in \mathfrak{H}^{\odot n}$, the $q$-th contraction between $f$ and $g$ is an element in $\mathfrak{H}^{\otimes (m+n-2q)}$ that is defined by
\begin{equation*}
 f \otimes_q g = \sum_{i_1,\dots,i_q=1}^{\infty} \langle f,e_{i_1} \otimes \cdots \otimes e_{i_q} \rangle_{\mathfrak{H}^{\otimes q}} \otimes \langle g,e_{i_1} \otimes  \cdots  \otimes e_{i_q} \rangle_{\mathfrak{H}^{\otimes q}} \,,
\end{equation*}
for $q=1,\dots, m \wedge n$.

For $g \in \mathfrak{H}^{\odot p}$ and $h \in \mathfrak{H}^{\odot q}$, we have the following product formula for the multiple integrals,
\begin{equation}\label{ito.prod}
	I_p(g) I_q(h) = \sum_{r=0}^{p \wedge q} r!\binom{p}{r} \binom{q}{r} I_{p+q-2r}(g\tilde\otimes_r h) \,,
\end{equation}
where $g\tilde\otimes_r h$ is the symmetrization of $g\otimes_r h$ (see \cite{Nou 12}).

The following Theorem \ref{fm.theorem}, known as the fourth moment theorem, provides necessary and sufficient conditions for the  convergence of a sequence of random variables to a normal distribution (see \cite{Nua Pec 05}).
\begin{thm} \label{fm.theorem}
 Let $ n \geq 2 $ be a fixed integer. Consider  a collection of elements $\{f_{T}, T>0\}$ such that $f_{T} \in \mathfrak{H}^{\odot n}$ for every $T>0$. Assume further that
 \[
  \lim_{T \to \infty}\mathbb{E} [I_n (f_T) ^2] = \lim_{T \to \infty} n! \|f_T\|^2_{\mathfrak{H}^{\otimes{n}}} = \sigma^2 .
 \]
Then the following conditions are equivalent:
 \begin{enumerate}
	\item $\lim_{T \to \infty} \mathbb{E}[I_n(f_T)^4] = { 3\sigma^4}$.
	\item For every $q=1,\dots, n-1$, $\lim_{T \to \infty}||f_T \otimes_q f_{T}||_{\mathfrak{H}^{\otimes 2(n-q)}} = 0$.
    \item As $T $ tends to infinity, the $n$-th multiple integrals $\{I_n (f_T), T \geq 0\}$ converge in distribution to a Gaussian random variable $N(0,\sigma^2)$.
 \end{enumerate}	
\end{thm}
The following theorem provides an estimate of the Kolmogrov distance between a nonlinear Gaussian functional and the standard normal random variable (see Corollary 1 of \cite{kim 3}).
\begin{thm}\label{kp}
Suppose that $\varphi_T(t,s)$ and $\psi_T(t,s)$ are two functions on $\mathfrak{H}^{\otimes 2}$.
Let $b_T$ be a positive function of $T$ such that $I_2(\psi_T)+b_T>0$ a.s.. Denote the functions $\Psi_i(T)$ as follows,
\begin{align*}
\Psi_1(T)&=\frac{1}{b_T^2}\sqrt{\big[b^2_T-2\norm{\varphi_T}_{\mathfrak{H}^{\otimes 2}}^2\big]^2+8\norm{\varphi_T \otimes_1 \varphi_T}_{\mathfrak{H}^{\otimes 2}}^2},\\
\Psi_2(T)&=\frac{2}{b_T^2}\sqrt{2\norm{\varphi_T \otimes_1 \psi_T}_{\mathfrak{H}^{\otimes 2}}^2+\innp{\varphi_T,\,\psi_T}_{\mathfrak{H}^{\otimes 2}}^2},
\end{align*}
and
$$\Psi_3(T)=\frac{2}{b_T^2}\sqrt{ \norm{\psi_T}_{\mathfrak{H}^{\otimes 2}}^4+2\norm{\psi_T \otimes_1\psi_T}_{\mathfrak{H}^{\otimes 2}}^2}.$$
{ Let $Z$ be a standard normal random variable.} If $\Psi_i(T)\to 0,\,i=1,2,3$ as $T\to \infty$, there exists a constant $c$ such that for $T$ large enough, 
\begin{equation}
\sup_{z\in \Rnum}\abs{P(\frac{I_2(\varphi_T)}{ I_2(\psi_T)+b_T}\le z)-P(Z\le z)}\le c\times \max_{i=1,2,3} \Psi_i(T).
\end{equation}
\end{thm}

\medskip
\section{Strong Consistency:  Proof of Theorem~\ref{thm strong} }

We first define some important functions that will be used in the proof. Denote
\begin{align}
f_T(t,s)&= e^{-\theta \abs{t-s}}\mathbbm{1}_{\set{0\le s,t\le T}},\label{ft ts 000} \\
h_T(t,s)&= e^{-\theta (T-t)-\theta (T-s)}\mathbbm{1}_{\set{0\le s,t\le T}},\label{ht ts}\\
g_T(t,s)&= \frac{1}{2\theta T}(f_T- h_T).\label{gt ts}
\end{align} 
The solution to the SDE \eqref{fOU} with $\sigma=1$ is $$X_t = \int_0^t e^{-\theta(t-s)} dG_s = I_1(f_T(t, \cdot) \mathbbm{1}_{[0,t]}(\cdot))\,.$$
We apply the product formula of multiple integrals \eqref{ito.prod} and stochastic Fubini theorem to obtain
\begin{align}
\frac{1}{T} \int_0^T X_t^2\mathrm{d} t =  I_2(g_T)+b_T,  \label{x square}
\end{align}
where 
 \begin{equation}
	 b_T=\frac{1}{T}\int_0^T\, \norm{e^{-\theta (t-\cdot) }\mathbbm{1}_{[0,t]}(\cdot)}^2 _{\mathfrak{H}}\dif t.\label{bt bt}
 \end{equation}
From the equation (\ref{hattheta}), we can write
\begin{align} \label{ratio 1}
\sqrt{T} (\hat{\theta}_T-\theta )&=-\frac{\frac{1}{2\sqrt{T}} I_2(f_T)}{ I_2(g_T)+b_T}.
\end{align}

 In the remaining part of this paper, $C$ will be a generic positive constant independent of $T$ whose value may differ from line to line.
 
\begin{nott}
For a function $\phi(r) \in \mathcal{V}_{[0,T]}$, we define two norms as 
\begin{align}\label{norms}
 \norm{\phi}_{\FH_1}^2&=C_{\beta}\int_{[0,\,T]^2} \phi(r_1)\phi(r_2) \abs{r_1-r_2}^{2\beta-2} \dif r_1\dif r_2,\\
 \norm{\phi}_{\FH_2}^2&=C_{\beta}' \int_{[0,\,T]^2} \abs{\phi(r_1)\phi(r_2)} ({r_1 r_2}  )^{\beta-1} \dif r_1\dif r_2.
 \end{align}
For a function $\varphi(r,s)$ in $[0,\,T]^2$, define an operator from $ (\mathcal{V}_{[0,T]})^{\otimes 2}$ to $\mathcal{V}_{[0,T]}$ as follows,
\begin{align}\label{ft sup}
(\tensor{K}\varphi) (r)=\int_{0}^{T} \abs{\varphi(r,u)} u^{   {\beta-1 }}\dif u.
\end{align}
\end{nott}
\begin{remark} If $C_{\beta}'>0 $, the norm $ \norm{\cdot}_{\FH_2}$ is equivalent to the standard norm in  $L^1([0,T],\nu)$ with $\nu(\dif x)=x^{\beta-1}\dif x$.
\end{remark}

The following proposition is a consequence of the identity (\ref{innp fg3}).
\begin{prop}
Suppose that Hypothesis~\ref{hypthe 1} holds. Then for any $\phi\in \mathcal{V}_{[0,T]}$,
 \begin{equation}\label{norm.ineq2}
 	\abs{\norm{\phi}_{\FH }^2 - \norm{\phi}_{\FH_1 }^2 }\leq \norm{\phi}_{\FH_2 }^2,
 \end{equation}
and for any $ \varphi, \psi \in (\mathcal{V}_{[0,T]})^{\odot 2}$,
 \begin{align}
 	\abs{\norm{\varphi}_{\FH^{\otimes 2}}^2 - \norm{\varphi}_{\FH_1^{\otimes 2}}^2 }&\leq \norm{\varphi}_{\FH_2^{\otimes 2}}^2 + 2C_{\beta}' \norm{\tensor{K}\varphi}_{\FH_1}^2,\label{norm.ineq}\\
	\abs{\innp{\varphi,\, \psi}_{\FH^{\otimes 2}}  - \innp{\varphi,\,\psi}_{\FH_1^{\otimes 2}} } & \leq  \abs{\innp{\varphi,\,\psi}_{\FH_2^{\otimes 2}}}  + 2C_{\beta}' \abs{\innp{\tensor{K}\varphi,\, \tensor{K}\psi}_{\FH_1}}. \label{inner product 00.ineq}
 \end{align}
\end{prop}

For any $t\in [0,T]$, denote the Wiener-It\^{o} stochastic integral of $f_T(r,s)\mathbbm{1}_{\set{0\le r,s\le t}}$ as 
\begin{equation}\label{Ft}
	F_t:=I_2(f_T(r,s)\mathbbm{1}_{\set{0\le r,s\le t}}).
\end{equation}
The next two propositions are about the asymptotic behaviors of the second moment of $F_T$ and the increment $F_t - F_s$ with $0\le t,s\le T$ respectively.
First, we need a technical lemma.
\begin{lemma} \label{upper bound F}
 Assume $\beta\in (0,1)$. There exists a constant $C>0$ such that for any  $s\in [0,\infty)$,
\begin{equation}
e^{-\theta s}\int_0^{s} e^{\theta r} r^{\beta -1}\dif r \le C (1\wedge s^{\beta-1}).
\end{equation}
\end{lemma}
\begin{proof}
It is easy to check that the following function 
$$A(s)= e^{-\theta s}\int_0^{s} e^{\theta r} r^{\beta -1}\dif r,\qquad s\in [0,\infty)$$ 
is continuous and $\lim_{s\to \infty} A(s)=0$. Then $A(s)$ is bounded on $[0,\infty)$. In addition, we have
\begin{align*}
\lim_{s\to\infty}\frac{A(s)}{s^{\beta-1}}=\lim_{s\to\infty}\frac{\int_0^{s} e^{\theta r} r^{\beta -1}\dif r}{e^{\theta s}s^{\beta-1}}=\frac{1}{\theta} \,,
\end{align*}  
from L'H\^opital's rule, and clearly
\begin{align*}
\lim_{s\to 0}\frac{A(s)}{s^{\beta-1}}=0 \,,
\end{align*}  
so $|A(s)| \leq C s^{\beta-1}$. Hence we obtain the conclusion.
\end{proof}
\begin{prop}\label{Ft 2 norm}
When $\beta\in (\frac12,\,\frac34)$,
\begin{align}\label{Ft-norm}
\lim_{T\to\infty}\frac{1}{4 \theta \sigma_{\beta}^2 T}\E[\abs{F_T} ^2]= (C_\beta \Gamma(2\beta-1) \theta^{-2\beta})^2 .
\end{align}
When  $\beta=\frac34$,
\begin{align}\label{Ft-norm 2}
\lim_{T\to\infty}\frac{1}{T\log T}\E[\abs{F_T} ^2]= 4C_{\beta}^2\theta^{-2 } .
\end{align}
When  $\beta\in (\frac34, \,1)$,
\begin{align}\label{Ft-norm 3}
\limsup_{T\to\infty}\frac{1}{ T^{4\beta-2}}\E[\abs{F_T} ^2]<\infty .
\end{align}
\end{prop}
\begin{proof}
By It\^{o}'s isometry, we have
\begin{align*} 
\E[\abs{F_T}^2]&=2\norm{f_T}^2_{\FH^{\otimes 2}}.
\end{align*}
The inequality (\ref{norm.ineq}) implies that 
\begin{align}\label{ft norm 02}
\abs{ \norm{f_T}^2_{\FH^{\otimes 2}}-\norm{ f_T}_{\FH_1^{\otimes 2}}^2} \le  \norm{f_T }_{\FH_2^{\otimes 2}}^2+2C_{\beta}'\norm{ \tensor{K}f_T}_{\FH_1}^2.
\end{align}
First, Lemma 5.3 in \cite{hu Nua 10}  implies that when $\beta\in (\frac12,\,\frac34)$,
\begin{align}\label{fth1}
\lim_{T\to \infty } \frac{1}{2 \theta \sigma_{\beta}^2 T}\norm{ f_T}_{\FH_1^{\otimes 2}}^2&=(C_\beta   \Gamma(2\beta-1 )\theta^{-2\beta}  )^2.
\end{align}
Moreover,  we have
\begin{align}
\int_{[0,T]^2}e^{-\theta\abs{r-s}}(rs)^{\beta-1}\dif r\dif s &=2\int_{0<s<r\le T} e^{-\theta(r-s)}(rs)^{\beta-1}\dif r\dif s \nonumber\\
&\le 2\int_0^T s^{2(\beta-1)} (\int_{s}^T e^{-\theta(r-s)}\dif r)\dif s \nonumber \\
&\le C T^{2\beta -1},\label{upper bound 0011}
\end{align} so
\begin{align}\label{ft fh2 norm 2}
\norm{f_T }_{\FH_2^{\otimes 2}}^2 =\left| C_{\beta}' \int_{[0,T]^2}e^{-\theta\abs{r-s}}(rs)^{\beta-1}\dif r\dif s\right|^2\le C T^{4\beta -2}.
\end{align}
Meanwhile, we have
 $$\tensor{K}f_T (r) = \Big[\int_{0}^r e^{-\theta(r-u)}u^{\beta-1}\dif u +\int_r^T e^{-\theta( u-r)}u^{\beta-1}\dif u\Big] \mathbbm{1}_{\set{0\le r \le T}} \,.$$
Lemma~\ref{upper bound F} and making change of variable $v=u-r$ yield
\begin{align}
 \tensor{K}f_T (r) & \le \Big[C r^{\beta -1}+\int_0^{\infty} e^{-\theta v} (v+r)^{\beta-1}\dif v\Big] \mathbbm{1}_{\set{0\le r \le T}} \nonumber  \\
 &\le C r^{\beta -1} \mathbbm{1}_{\set{0\le r \le T}}. \label{Kft bound}
\end{align}
Then
\begin{align}\label{Kft bound 02}
\norm{\tensor{K}f_T}_{\FH_1}^2 & \le C \int_{0<v<u<T} (u-v)^{2\beta -2} (uv)^{\beta-1}\dif u\dif v = C T^{4\beta -2}.
\end{align}
Combining the above inequalities with (\ref{ft norm 02}), we obtain that  when $\beta\in (\frac12,\,\frac34)$,
	$$\lim_{T\to\infty}\frac{1}{2 \theta \sigma_{\beta}^2 T} \abs{\norm{f_T}^2_{\FH^{\otimes 2}}-\norm{f_T}_{\FH_1^{\otimes 2}}^2} =0 , $$
which together with \eqref{fth1} and {It\^{o}'s isometry} implies the desired (\ref{Ft-norm}).
In the same way, we can obtain (\ref{Ft-norm 2}) and (\ref{Ft-norm 3}) from Lemma~ 17 of \cite{hu nua zhou 19}.
\end{proof}
\begin{prop}\label{Ft 2 norm 1} 
If Hypothesis~\ref{hypthe 1} is satisfied,
there exists a  constant $C>0$ independent of $T$ such that for all $s,t\ge 0$, 
\begin{align}\label{E norm 2}
\E[\abs{F_t-F_s} ^2]\le C \big[ \abs{t-s}^{4\beta}+ \abs{t-s}^{2\beta}+\abs{t-s}^{2\beta-1} \big].
\end{align}
\end{prop}
\begin{remark}
Although the inner product of the Hilbert space $\FH$ is related to $T$, the constant $C$ in the above proposition does not depend on $T$. This fact is crucial to the proof of Proposition~\ref{prop ft ht}.
\end{remark}
\begin{proof} Let $0\le s<t\le T$.
 It\^{o}'s isometry implies that 
\begin{align}\label{ft fs}
\E[\abs{F_t-F_s} ^2]&=2\norm{f_t-f_s}^2_{\FH^{\otimes 2}} \le 4 (\norm{\phi_1}_{\FH^{\otimes 2}}^2+\norm{\phi_2 }_{\FH^{\otimes 2}}^2),
\end{align}
where 
\begin{align*}
\phi_1(r_1,r_2)&=e^{-\theta \abs{r_1-r_2}}\mathbbm{1}_{\set{s\le r_1,r_2\le t}},\\
\phi_2(r_1,r_2)&=e^{-\theta \abs{r_1-r_2}}(\mathbbm{1}_{\set{0\le r_1\le s,\, s\le r_2\le t}} + \mathbbm{1}_{\set{0\le r_2\le s,\, s\le r_1\le t}} ).
\end{align*}
Clearly, we have
\begin{align}
\norm{ \phi_1}_{\FH_1^{\otimes 2}}^2&=C_{\beta}^2 \int_{[s,t]^4}e^{-\theta \abs{r_1-r_2}}e^{-\theta \abs{u_1-u_2}} \abs{r_1-u_1}^{2\beta -2} \abs{r_2-u_2}^{2\beta -2}\dif \vec{u}\dif \vec{r} \nonumber\\
&\le C_{\beta}^2 \int_{[s,t]^4} \abs{r_1-u_1}^{2\beta -2} \abs{r_2-u_2}^{2\beta -2}\dif \vec{u}\dif \vec{r} \nonumber\\
&= \frac{C_{\beta}^2}{\big((2\beta-1)\beta \big)^2} \abs{t-s}^{4\beta}, \label{ph1 h1 norm}
\end{align} 
and 
\begin{align}
\norm{ \phi_1}_{\FH_2^{\otimes 2}}^2 &=C_{\beta}'^2 \int_{[s,t]^4}e^{-\theta \abs{r_1-r_2}}e^{-\theta \abs{u_1-u_2}} (r_1 u_1 r_2 u_2)^{\beta -1}  \dif \vec{u}\dif \vec{r} \nonumber\\
& \le C_{\beta}'^2 \big(\int_s^t r^{\beta -1}\dif r  \big)^4 = \frac{C_{\beta}'^2}{\beta^4}  (t^{\beta}-s^{\beta})^4 = \frac{C_{\beta}'^2}{\beta^4} t^{4\beta} \left(1-\left(\frac{s}{t}\right)^{\beta}\right)^4 \nonumber\\
&\le \frac{C_{\beta}'^2}{\beta^4} t^{4\beta} \left(1-\frac{s}{t}\right)^{4\beta} = \frac{C_{\beta}'^2}{\beta^4} \abs{t-s}^{4\beta}, \label{ph1 h2 norm}
\end{align} where in the last inequality we have used the fact $1-x^\beta \leq (1-x)^\beta$ for any $x \in [0, 1]$. Similarly, we have
\begin{align*}
\tensor{K}\phi_1(r)&=\mathbbm{1}_{\set{s\le r \le t}}\int_{s}^{t} e^{-\theta\abs{r-u}}u^{\beta-1}\dif u \le \frac{1}{\beta} (t-s)^{\beta}\mathbbm{1}_{\set{s\le r \le t}},
\end{align*}  which implies
\begin{eqnarray}
  \norm{\tensor{K}\phi_1}^2_{\FH_1} \le \frac{1}{\beta^2}(t-s)^{2\beta} \norm{\mathbbm{1}_{[s,t]}(\cdot)}^2_{\FH_1} \le \frac{1}{\beta^2}(t-s)^{2\beta} \cdot C_\beta \int_{[s,t]^2} |u-v|^{2\beta-2} du dv \nonumber\\
  \leq \frac{C_{\beta}}{2\beta^3(2\beta-1)} \abs{t-s}^{4\beta}. \label{kph1 h1 norm}
\end{eqnarray}
Hence, by \eqref{ph1 h1 norm}, \eqref{ph1 h2 norm}, \eqref{kph1 h1 norm}, it follows from the inequality (\ref{norm.ineq}) that 
\begin{equation}\label{phi1 norm}
\norm{ \phi_1}_{\FH^{\otimes 2}}^2\le \Big( \frac{C_{\beta}^2}{(2\beta-1)^2} +\frac{C_{\beta}'^2}{\beta^2}+\frac{C_{\beta}} {\beta(2\beta-1)} \Big)\frac{ \abs{t-s}^{4\beta}}{\beta^2}.
\end{equation}
Next, by the symmetry of the function $\phi_2$ we have
\begin{align*}
\norm{ \phi_2}_{\FH_1^{\otimes 2}}^2& = 16 C_{\beta}^2\int_{0\le r_1\le u_1\le s\le r_2\le u_2\le t}e^{-\theta \abs{r_1-r_2}}e^{-\theta \abs{u_1-u_2}} \abs{r_1-u_1}^{2\beta -2} \abs{r_2-u_2}^{2\beta -2}\dif \vec{u}\dif \vec{r}. 
\end{align*}
Making the change of variables $a=u_1-r_1,\, b=s-u_1,\, c=r_2-s,\, p =u_2-r_2 $, we have
\begin{align*}
&\quad \int_{0\le r_1\le u_1\le s\le r_2\le u_2\le t}e^{-\theta \abs{r_1-r_2}}e^{-\theta \abs{u_1-u_2}} \abs{r_1-u_1}^{2\beta -2} \abs{r_2-u_2}^{2\beta -2}\dif \vec{u}\dif \vec{r}\\
&=  \int_{0}^{s}e^{-\theta a} a^{2\beta -2} \dif a \int_{0}^{s-a}e^{-2\theta b} \dif b\int_{0}^{t-s} e^{-\theta p} p^{2\beta -2}\dif p \int_{0}^{t-s-p} e^{-2\theta c}\dif c\\
& \le \frac{\Gamma(2\beta-1)}{4\theta^{2\beta+1}} \int_{0}^{t-s}  p^{2\beta -2}\dif p =\frac{\Gamma(2\beta-1)}{4\theta^{2\beta+1}(2\beta-1)}  \abs{t-s}^{2\beta-1},
\end{align*}which implies that 
\begin{align}\label{phi2 H1}
\norm{\phi_2}_{\FH_1^{\otimes 2}}^2\le \frac{4C_{\beta}^2 \Gamma(2\beta-1)}{(2\beta-1)\theta^{2\beta+1}} \abs{t-s}^{2\beta-1}.
\end{align}
The symmetry also implies that
\begin{align}
\norm{ \phi_2}_{\FH_2^{\otimes 2}}^2&= 2 \norm{e^{-\theta \abs{r_1-r_2}}\mathbbm{1}_{\set{0\le r_1\le s,\, s\le r_2\le t}}}^2_{\FH_2^{\otimes 2}}\nonumber \\
&=2\Big( C_{\beta}' \int_{0\le r_1 \le s} e^{\theta  r_1  }r_1 ^{\beta -1} \dif  {r}_1 \int_s^{t} e^{-\theta r_2} r_2^{\beta -1}\dif r_2 \Big)^2 .\label{phi2 fh2}
\end{align}
Making the change of variables $r_2=u+s$ implies that 
\begin{align}
\int_s^{t} e^{-\theta r_2} r_2^{\beta -1}\dif r_2 &=\int_{0}^{t-s} e^{-\theta (u+s)} (u+s)^{\beta-1}\dif u\nonumber \\
& \le e^{-\theta s}\int_{0}^{t-s}  u^{\beta-1}\dif u=\frac{1}{\beta}e^{-\theta s}(t-s)^{\beta}. \label{phi2 fh2 1}
 \end{align}
Substituting (\ref{phi2 fh2 1}) into the identity (\ref{phi2 fh2}) and then using Lemma~\ref{upper bound F},  we have 
 \begin{align}\label{phi2 H2}
 \norm{\phi_2}_{\FH_2^{\otimes 2}}^2&\le 2\big(\frac{ C C_{\beta}'}{\beta}\big)^2 \abs{t-s}^{2\beta}.
 \end{align}
Using Lemma~\ref{upper bound F} again, we obtain
\begin{align*}
\tensor{K}\phi_2(r)&=\mathbbm{1}_{[0,s]}(r) \int_{s}^{t} e^{-\theta(u-r)}u^{\beta-1}\dif u+\mathbbm{1}_{[s,t]}(r) \int_{0}^{s} e^{-\theta(r-u)}u^{\beta-1}\dif u\\
&\le  \frac{1}{\beta}(t-s)^{\beta}e^{-\theta(s-r)}\mathbbm{1}_{[0,s]}(r) + C { e^{-\theta( r-s)}\mathbbm{1}_{[s,t]}(r)} .
\end{align*}
Therefore,
\begin{eqnarray*}
	\norm{\tensor{K}\phi_2}^2_{\FH_1} \le \frac{1}{\beta^2}(t-s)^{2\beta} \norm{e^{-\theta(s-\cdot)}\mathbbm{1}_{[0,s]}(\cdot)}^2_{\FH_1} + C^2 \norm{e^{-\theta(\cdot-s)}\mathbbm{1}_{[s,t]}(\cdot)}^2_{\FH_1} =: I_1 + I_2 .
\end{eqnarray*}
For the term $I_1$,
\begin{eqnarray*}
	I_1 & = & \frac{C_{\beta}}{\beta^2}(t-s)^{2\beta} \int_{[0,s]^2}e^{-\theta(s-r_1)-\theta(s-r_2)} \abs{r_1-r_2}^{2\beta-2}\dif r_1\dif r_2 \nonumber \\
	& = & \frac{2C_{\beta}}{\beta^2}(t-s)^{2\beta} \int_{[0,s]^2, r_1 \leq r_2}e^{-\theta(r_1 + r_2)} \abs{r_2-r_1}^{2\beta-2}\dif r_1\dif r_2 \\
	& = & \frac{2C_{\beta}}{\beta^2}(t-s)^{2\beta}  \int_{[0,s]^2,\,0 \le a+b\le s}e^{-\theta(a+2b)} a^{2\beta-2}\dif a\dif b \leq \frac{C_{\beta} \Gamma(2\beta-1)(t-s)^{2\beta}}{\theta^{2\beta} \beta^2},\label{uper bound 000}
\end{eqnarray*}
where we have made the change of variables, $a=r_2 - r_1$ and $b=r_1$. Similarly, for the term $I_2$,
\begin{eqnarray*}
    I_2 & = & C^2 C_{\beta} \int_{[s,t]^2}e^{-\theta(r_1-s)-\theta(r_2-s)} \abs{r_1-r_2}^{2\beta-2}\dif r_1\dif r_2 \\ 
	& = & 2 C^2 C_{\beta} \int_{[0,t-s]^2,\, r_1 \leq r_2}e^{-\theta(r_1 + r_2)} \abs{r_2-r_1}^{2\beta-2}\dif r_1\dif r_2 \\
	& = & 2 C^2 C_{\beta} \int_{[0,t-s]^2,\,0  \leq  a+b\le t- s}e^{-\theta(a+2b)} a^{2\beta-2}\dif a\dif b \leq \frac{2C^2 C_{\beta}(t-s)^{2\beta-1}}{2\beta-1}. \nonumber
\end{eqnarray*}
Hence, we have
\begin{equation}\label{Kphi2 H1}
	\norm{\tensor{K}\phi_2}^2_{\FH_1} \le C_{\beta} \left(\frac{\Gamma(2\beta-1)}{\theta^{2\beta} \beta^2}\abs{t-s}^{2\beta}+ \frac{2C^2}{2\beta-1}\abs{t-s}^{2\beta-1}\right).
\end{equation}
By the inequalities (\ref{norm.ineq}), \eqref{phi2 H1}, \eqref{phi2 H2}, \eqref{Kphi2 H1}, there exists a constant $C'>0$ independent of $T$ such that $$\norm{\phi_2}_{\FH^{\otimes 2}}^2\le C' \big[  \abs{t-s}^{2\beta}+\abs{t-s}^{2\beta-1} \big].$$ 
Combining it with (\ref{ft fs}) and (\ref{phi1 norm}), we obtain the desired (\ref{E norm 2}).
\end{proof}
For any $t\in [0,T]$, we denote 
\begin{equation}\label{ht}
	H_t=I_2(h_T(r,u)\mathbbm{1}_{\set{0\le r,u\le t}}),
\end{equation}
and we apply the similar computations as above to obtain the following results about asymptotic behavior of $H_T$ and the increment $H_t - H_s$.

\begin{prop}\label{prop ht} If Hypothesis~\ref{hypthe 1} is satisfied, there exists a constant $C>0$ independent of $T$ such that 
\begin{align*}
\sup_{t\ge 0}\E[\abs{H_t} ^2]&< C,
\end{align*} 
and there exist two constants $C'>0$ and $\alpha\in (0,1)$ independent of $T$ such that for any $\abs{t-s}\le 1$,
\begin{align*}
\E[\abs{H_t-H_s} ^2]&\le C' \abs{t-s}^{\alpha}.
\end{align*}
\end{prop}
\begin{prop}\label{prop ft ht} Let $F_T$ and $H_T$ be given in \eqref{Ft} and \eqref{ht} respectively. If Hypothesis~\ref{hypthe 1} is satisfied,
we have
\begin{align}
\lim_{T\to \infty}\frac{F_T}{T}=0 \quad {\rm and} \quad \lim_{T\to \infty}\frac{H_T}{ {T}}=0, \qquad a.s..
\end{align}
\end{prop}
\begin{proof}
The proof is similar as \cite{chw 17}. We will only show $\lim_{T\to \infty}\frac{F_T}{T}=0$, and the other is similar. First,  when $\beta \in (\frac{1}{2}, \frac{3}{4}]$, Chebyshev's inequality, the hypercontractivity of multiple Wiener-It\^{o} integrals and Proposition~\ref{Ft 2 norm} imply that for any $\epsilon > 0$, 
\begin{equation*}P\left(\frac{F_n}{n} > \epsilon \right) \leq \frac{\mathbb{E}F_n^4}{n^4 \epsilon^4} \leq \frac{C\left(\mathbb{E}(F_n^2)\right)^2}{n^4 \epsilon^4}
 \leq C n^{-2}\log n \,.
 \end{equation*}When $\beta \in (\frac34, 1)$, we take an integer $p>\frac{1}{2(1-\beta)}$. Then we apply  Chebyshev's inequality to obtain that for any $\epsilon > 0$,
 \begin{equation*}P\left(\frac{F_n}{n} > \epsilon \right) \leq \frac{\mathbb{E}F_n^p}{n^p \epsilon^p} \leq \frac{C\left(\mathbb{E}(F_n^2)\right)^{p/2}}{n^p \epsilon^p}
 \leq \frac{C}{n^{2p(1-\beta )}}   \,. 
 \end{equation*} 
The Borel-Cantelli lemma implies for $\beta \in (\frac{1}{2}, 1)$,
\begin{align*}
\lim_{n\to \infty}\frac{F_n}{n}=0, \ a.s..
\end{align*} 

Second, Propositions~\ref{Ft 2 norm 1} implies that there exist two constants $\alpha\in (0,1 ),\,{ C_{\alpha, \beta}}>0$ independent of $T$  such that for any $\abs{t-s}\le 1$,
\begin{equation*}
\E[\abs{F_t-F_s}^2] \le {C_{\alpha, \beta}} \abs{t-s}^ {\alpha}.
\end{equation*}
Then the Garsia-Rodemich-Rumsey inequality implies that for any real number $p > \frac{4}{\alpha}, q >1$ and integer $n \ge1$, 
\begin{align*}
\abs{F_t-F_s}\le R_{p,q} n^{q/p} , \qquad \forall \ t,s\in [n,n+1] ,
\end{align*}where $R_{p,q}$ is a random constant independent of $n$ (see Proposition 3.4 of \cite{chw 17}).

Finally, since
\begin{align*}
\abs{\frac{F_T}{T}}\le \frac{1}{T}\abs{F_T-F_n}+ \frac{n}{T}\frac{\abs{F_n}}{n},
\end{align*}where $n=[T]$ is the biggest integer less than or equal to a real number $T$, we have $\frac{F_T}{T}$ converges to $0$ almost surely as $T\to \infty$.
\end{proof}

Proposition \ref{prop ft ht} implies $I_2(g_T) = \frac{F_T - H_T}{2\theta T} \to 0$ as $T \to \infty$ almost surely. Next we study the term $b_T$. 
\begin{prop}\label{prop bt lim} Let $b_T$ be given by \eqref{bt bt}. Suppose that Hypothesis~\ref{hypthe 1} holds. We have
\begin{align}
\lim_{T\to\infty}b_T = C_{\beta}  \Gamma(2\beta-1) \theta^{-2\beta}   >0.
\end{align}
\end{prop}
\begin{proof}
From Lemma~\ref{upper bound F}, we obtain
\begin{align}
	 \norm{e^{-\theta (t-\cdot) }\mathbbm{1}_{[0,t]}(\cdot)}^2 _{\mathfrak{H}_2}& =C'_{\beta} \big(\int_{0}^t\,e^{-\theta(t-u)} u^{\beta-1}\dif u\big)^2\nonumber\\ 
	&\le C'_{\beta} C^2 t^{2(\beta-1)}, \label{bt estm}
\end{align} 
which together with Hypothesis~\ref{hypthe 1} stating $|\Psi(r,s)| \leq C_\beta'|rs|^{\beta-1}$, implies
\begin{align*}
\lim_{T\to\infty} \abs{\int_{[0,\,T]^2}e^{-\theta (T-r) -\theta (T-s)}\Psi(r,s) \dif r \dif s} \leq \lim_{T\to\infty} \abs{\int_0^T e^{-\theta (T-r)} r^{\beta-1} \dif r}^2 = 0.
\end{align*}
Then the identity (\ref{innp fg3}) implies that 
\begin{align*}
\lim_{T\to\infty}b_T&= \lim_{T\to\infty} \norm{e^{-\theta (T-\cdot) }\mathbbm{1}_{[0,T]}(\cdot)}^2 _{\mathfrak{H}}\\
&=   \lim_{T\to\infty} \norm{e^{-\theta (T-\cdot) }\mathbbm{1}_{[0,T]}(\cdot)}^2 _{\mathfrak{H}_1}+ \lim_{T\to\infty} \abs{\int_{[0,\,T]^2}e^{-\theta (T-r) -\theta (T-s)}\Psi(r,s) \dif r \dif s}\\
&=  C_{\beta}\Gamma(2\beta-1)\theta^{-2\beta},
\end{align*}
where the last step is from \cite{hu Nua 10} or  \cite{chenkl 19}.
\end{proof}
\begin{remark} \label{rem 36}
The upper bound (\ref{bt estm}) implies that as $T\to\infty$, the speed of convergence 
$$ \frac{1}{T}\int_0^T \dif t \int_{[0,\,t]^2}e^{-\theta (t-r) -\theta (t-s)}\Psi(r,s) \dif r \dif s\to 0$$ 
is at least $\frac{1}{T^{2(1-\beta)}}$. Lemma 3.2 of \cite{chenkl 19} implies that the speed of convergence 
$$\frac{1}{T}\int_0^T  \norm{e^{-\theta (t-\cdot) }\mathbbm{1}_{[0,t]}(\cdot)}^2 _{\mathfrak{H}_1}  \dif t \to  C_{\beta}\Gamma(2\beta-1)\theta^{-2\beta} $$ is at least $\frac{1}{T}$. By the identity (\ref{innp fg3}), there exists a constant $C>0$ such that when $T$  is large enough, 
	$$\abs{b_T- C_{\beta}\Gamma(2\beta-1)\theta^{-2\beta}} \leq \frac{C}{T^{2(1-\beta)}} .$$
\end{remark}

\noindent{\it Proof of Theorem~\ref{thm strong}.\,} From \eqref{gt ts}, \eqref{x square}, \eqref{Ft}, and \eqref{ph1 h2 norm},
\begin{align*}
\frac{1}{T} \int_0^T X_t^2\mathrm{d} t&=  I_2(g_T)+b_T=\frac{1}{2\theta}[\frac{F_T}{T} - \frac{H_T}{T}]+ b_T .
\end{align*} 
Proposition~\ref{prop ft ht} and \ref{prop bt lim} imply that 
\begin{equation*}
\lim_{T\to \infty}\frac{1}{T} \int_0^T X_t^2\mathrm{d} t = C_{\beta} \Gamma(2\beta-1) \theta^{-2\beta},\ a.s.,
\end{equation*}which implies that the second moment estimator $\tilde{\theta}_T$ is strongly consistent.

Since \begin{align*} \label{ratio 1}
 \hat{\theta}_T-\theta &=\frac{- \frac{1}{2T} F_T}{ I_2(g_T)+b_T}, 
\end{align*}
 Proposition~\ref{prop ft ht} and \ref{prop bt lim} imply that the least squares estimator $\hat{\theta}_T$ is also strongly consistent.
{\hfill\large{$\Box$}}

\section{The Asymptotic normality}

\begin{prop}\label{contraction ft} Let $\gamma $ be given as in Theorem~\ref{B-E bound thm}. 
When $\beta\in (\frac12,\,\frac34)$, there exists a constant $C_{\theta,\,\beta} > 0$ such that 
\begin{equation}\label{zhou ineq}
\frac{1}{T}\norm{f_T\otimes_{1} f_T}_{\mathfrak{H}^{\otimes 2}}\le 
\frac{C_{\theta,\,\beta}}{T^{\gamma}}.
\end{equation}
\end{prop}
\begin{proof} Without loss of generality, we assume $C_{\beta}=C_{\beta}'=1$.
	Recall that
	\begin{align*}
	\left(f_T\otimes_1 f_T\right) (u_1,u_2) &:= \int_{[0,T]^2} f_T(u_1,v_1) f_T(u_2,v_2)\, \frac{\partial ^2 }{\partial v_1\partial v_2}R(v_1,\,v_2) \dif v_1\dif v_2.
	\end{align*}
	Denote
	$$\left(f_T\otimes_{1'} f_T\right) (u_1,u_2) :=  \int_{[0,T]^2} f_T(u_1,v_1) f_T(u_2,v_2)\, \abs{v_1- v_2}^{2\beta-2} \dif v_1\dif v_2.$$
Clearly, the functions $f_T$, $f_T \otimes_{1'} f_T $ and $ \tensor{K}f_T \otimes \tensor{K}f_T$ are positive on $[0,T]^2$, and
$$ \abs{f_T \otimes_1 f_T} \le 
 { C_\beta} \abs{f_T \otimes_{1'} f_T} + {  C'_\beta} \abs{\tensor{K}f_T \otimes \tensor{K}f_T}.$$
This implies
\begin{equation}
\|f_T \otimes_1 f_T \|^2_{\mathfrak{H}^{\otimes 2}} \le 2{  C}\big[  \norm{f_T \otimes_{1'} f_T}^2_{\mathfrak{H}^{\otimes 2}} + \norm{\tensor{K}f_T \otimes \tensor{K}f_T}^2_{\mathfrak{H}^{\otimes 2}}\big].\label{ft contraction sum}
\end{equation}
We first deal with the first item on the right hand side of \eqref{ft contraction sum}.  The inequality (\ref{norm.ineq}) implies that
\begin{align}\label{ft contraction 01}
\abs{ \norm{f_T \otimes_{1'} f_T}^2_{\mathfrak{H}^{\otimes 2}}- \norm{f_T \otimes_{1'} f_T}^2_{\mathfrak{H}_1^{\otimes 2}}}&\le \norm{f_T \otimes_{1'} f_T}^2_{\mathfrak{H}_2^{\otimes 2}}
 +2 C'_\beta  \norm{\tensor{K}(f_T \otimes_{1'} f_T)}^2_{\FH_1}.
\end{align}
By Theorem 5 in \cite{hu nua zhou 19} and its proof, and Lemma 5.4 of \cite{hu Nua 10} (see the archive version), we have
\begin{equation}\label{zhou ineq 02}
\frac{1}{T}\norm{f_T\otimes_{1'} f_T}_{\mathfrak{H}_1^{\otimes 2}}\le \frac{C_{\theta,\,\beta}}{T^{\gamma}} \,.
\end{equation}
Lemma 3.6 of \cite{chw 17}  also implies
 \begin{align*}
 \big(f_T \otimes_{1'} f_T \big)(u,v) \le C \abs{u-v}^{2\beta-2}\mathbbm{1}_{[0,\,T]^2}(u,v).
 \end{align*}
 As a result,
   \begin{equation}\label{ft tensor 1p}
	   \norm{f_T \otimes_{1'} f_T}^2_{\mathfrak{H}_2^{\otimes 2}} \leq C  \Big( \int_{[0,T]^2}\abs{u-v}^{2\beta-2} (uv)^{\beta -1}\dif u\dif v \Big)^2 = C T^{2(4\beta-2)},
   \end{equation}
and
\begin{align*}
\tensor{K}(f_T \otimes_{1'} f_T)(u) &\le C\mathbbm{1}_{[0,\,T]}(u) \int_{0}^T  \abs{u-v}^{2\beta-2} v^{\beta -1}\dif v\\
&=C\mathbbm{1}_{[0,\,T]}(u) \Big[\int_{0}^u  ( u-v)^{2\beta-2} v^{\beta -1}\dif v + \int_{u}^T  (v-u)^{2\beta-2} v^{\beta -1}\dif v\Big] \\
&\le C\mathbbm{1}_{[0,\,T]}(u) \Big[u^{3\beta -2} + u^{\beta -1} \int_{u}^T (v-u)^{2\beta-2}\dif v\Big]\\
& \le C \mathbbm{1}_{[0,\,T]}(u) T^{2\beta -1} u^{\beta-1}.
\end{align*}
Hence,
 \begin{align}\label{Kft tensor 1p}
 \norm{\tensor{K}(f_T \otimes_{1'} f_T)}^2_{\FH_1} \le C  T^{2(2\beta -1)} \int _0^T u^{\beta-1}\dif u\int_0^T v^{\beta-1}|u-v|^{2\beta -2}\dif v= C T^{4(2\beta-1)}.
 \end{align}
 Therefore, when $\beta\in (\frac12,\,\frac34)$,
 \begin{align*}
\lim_{T\to\infty} \frac{1}{T^2}\big[\norm{f_T \otimes_{1'} f_T}^2_{\mathfrak{H}_2^{\otimes 2}}
 +\norm{\tensor{K}(f_T \otimes_{1'} f_T)}^2_{\FH_1} \big]=0.
\end{align*}
By (\ref{ft contraction 01})-(\ref{Kft tensor 1p}), we have 
\begin{align}\label{ft contraction 1}
 \frac{1}{T }\norm{f_T \otimes_{1'} f_T}_{\mathfrak{H}^{\otimes 2}}\le \max\left(\frac{C_{\theta,\,\beta}}{T^{\gamma}}, \ T^{4\beta-3}\right) = \frac{C_{\theta,\,\beta}}{T^{\gamma}}.
\end{align}

Next, we deal with the second item on the right hand side of \eqref{ft contraction sum}.  
The inequality (\ref{Kft bound}) implies that 
\begin{align*}
\norm{\tensor{K}f_T  }^2_{\mathfrak{H}_2 } \leq \int_0^T r^{\beta-1} u^{\beta-1} |ru|^{\beta-1} dr du = CT^{2(2\beta-1)},
\end{align*} which together with the inequalities (\ref{Kft bound 02}) and (\ref{norm.ineq2}) implies that 
	$$\norm{\tensor{K}f_T}_{\FH}^2\le  C T^{2(2\beta-1)} .$$
Therefore, we have
\begin{align*}
\norm{\tensor{K}f_T \otimes \tensor{K}f_T}^2_{\mathfrak{H}^{\otimes 2}}= \norm{\tensor{K}f_T}_{\FH}^4\le  C T^{4(2\beta-1)}.
\end{align*}
Hence, 
\begin{align*}
 \frac{1}{T^2}\norm{\tensor{K}f_T\otimes \tensor{K}f_T}^2_{\mathfrak{H}^{\otimes 2}}\le CT^{2(4\beta-3)},
\end{align*}which together with (\ref{ft contraction 1}) and (\ref{ft contraction sum}) implies the desired (\ref{zhou ineq}).

\end{proof}
\noindent{\it Proof of Theorem~\ref{main thm 2}.\,} Denote a constant that depends on $\theta$ and $\beta$ as $$a:= C_{\beta}  \Gamma(2\beta-1) \theta^{-2\beta}.$$
First, Proposition ~\ref{Ft 2 norm}, Proposition \ref{contraction ft} and Theorem \ref{fm.theorem} imply that as $T\to\infty$,
\begin{equation} \label{asy norm Ft}
\frac{1}{2\sqrt{T}} F_T  \stackrel{ {law}}{\to} \mathcal{N}(0,\,\theta a^2\sigma_{\beta}^2).
\end{equation}
Second, we rewrite the identity (\ref{ratio 1}) as 
\begin{align*}  
\sqrt{T}( \hat{\theta}_T-\theta) =\frac{-\frac{1}{2\sqrt{T}} F_T}{ I_2(g_T)+b_T}.
\end{align*}
The Slutsky's theorem, Proposition~\ref{prop bt lim}  and the convergence result (\ref{asy norm Ft})  imply that the asymptotic normality (\ref{asy norm LSE}) holds. 
Third,  we can show the asymptotic normality of
 \begin{align}\label{asy norm moment}
 \sqrt{T}\Big(\frac{1}{T} \int_0^T X_s^2\mathrm{d} s - a \Big)  \stackrel{ {law}}{\to} \mathcal{N}(0,\,a^2\sigma_{\beta}^2/\theta).
 \end{align}
 In fact, we have
 \begin{align}\label{asy norm moment 2}
 \sqrt{T}\Big(\frac{1}{T} \int_0^T X_s^2\mathrm{d} s -  a \Big)&= \frac{1}{2\theta}\Big[\frac{F_T}{\sqrt{T}} - \frac{H_T}{\sqrt{T}}\Big]+ \sqrt{T} \big(b_T-a) .
 \end{align}
Remark~\ref{rem 36} implies that the speed of convergence $b_T\to a$ is at least $\frac{1}{T^{2(1-\beta)}}$. Hence, when $\beta\in (\frac12,\,\frac34)$, $
 \lim_{T\to\infty} \sqrt{T} \big(b_T-a)=0.$
 Proposition~\ref{prop ht} and the proof of Proposition~\ref{prop ft ht} imply that $\frac{H_T}{\sqrt{T}}\to 0 $ a.s. as $T\to\infty$. Thus, the Slutsky's theorem implies that (\ref{asy norm moment}) holds. Finally, since 
\begin{align*}
 \tilde{\theta}_T=\Big( \frac{1}{C_{\beta}  \Gamma(2\beta-1)T} \int_0^T X_s^2\mathrm{d} s\Big)^{-\frac{1}{2\beta}} ,
\end{align*}
 the delta method (see \cite[Example 27.2]{Billing 08}) implies that the asymptotic normality (\ref{asy norm SME}) holds. 
{\hfill\large{$\Box$}}\\

\section{the Berry-Ess\'{e}en bound}
We need several lemmas before the proof of Theorem~\ref{B-E bound thm}.  Without loss of generality, we assume $C_{\beta}=C_{\beta}'=1$. 
The following lemma is a result of Lemma 3.5 in \cite{chenkl 19}  and the inequalities (\ref{norm.ineq}), (\ref{ft fh2 norm 2})-(\ref{Kft bound 02}).
\begin{lemma}\label{2ft2}
The speed of convergence  $$\frac{\norm{f_T}_{\FH^{\otimes 2}}^2}{2\theta \sigma^2_{\beta}T} \to \left( C_{\beta}  \Gamma(2\beta-1) \theta^{-2\beta} \right)^2$$ is $\frac{1}{T^{3-4\beta}}$ as $T\to \infty$.
\end{lemma}

\begin{lemma}\label{htto0}
Let $f_T,\, h_T$ be given in (\ref{ft ts 000}) and (\ref{ht ts}). There exists a constant $C>0$ independent of $T$ such that 
\begin{align*}
\norm{\frac{1}{\sqrt{T}} h_T}_{\mathfrak{H}^{\otimes 2}}^2\le \frac{C}{T},
\end{align*}
and 
\begin{align*}
\abs{\innp{\frac{1}{\sqrt{T}} f_T,\, \frac{1}{\sqrt{T}} h_T}_{\mathfrak{H}^{\otimes 2}}}\le \frac{C}{T}(1+ T^{3\beta-2}).
\end{align*}
\end{lemma}
\begin{proof}
It is equivalent to show there exists a constant $C>0$ independent of $T$ such that 
\begin{align} 
\norm{   h_T}_{\mathfrak{H}^{\otimes 2}}^2\le C, \label{ht to 011}\\
\abs{\innp{  f_T,\,  h_T}_{\mathfrak{H}^{\otimes 2}}} \le C(1+T^{3\beta-2}). \label{ht ft to 011}
\end{align}
First,  we have
\begin{align}
 \norm{h_T}_{\FH_1^{\otimes 2}}^2&=  \norm{e^{-\theta (T-\cdot)} \mathbbm{1}_{[0,T]}(\cdot)}_{\FH_1 }^4 \nonumber\\
&= \left( \int_{[0,T]^2}e^{-\theta[(T-t_1)+(T-t_2) ]}\abs{t_1-t_2}^{2\beta-2} \dif  {t_1}\dif  {t_2} \right)^2 \nonumber\\
&= 4\left( \int_{0<t_1<t_2<T}e^{-\theta[(T-t_1)+(T-t_2) ]} (t_2-t_1)^{2\beta-2} \dif  {t_1}\dif  {t_2} \right)^2 \nonumber\\
&=4\left( \int_{[0,T]^2,\,0<a+b<T}e^{-\theta(2a+b )}b^{2\beta-2} \dif  a\dif  b \right)^2 \nonumber\\
& \leq \left ( { \Gamma(2\beta -1)}{\theta^{-2\beta}}\right)^2, \label{ht h1}
\end{align}
where we have made the change of variables $T-t_2 \to a, \ t_2 - t_1 \to b$ in the fourth step.
We also have
\begin{align*}
 \norm{h_T}_{\FH_2^{\otimes 2}}^2&= \norm{e^{-\theta (T-\cdot)} \mathbbm{1}_{[0,T]}(\cdot)}_{\FH_2 }^4 = \left( \int_0^T e^{-\theta (T- r)}r^{\beta -1}\dif r  \right)^4 \le C,
\end{align*} where for the inequality we have used Lemma~\ref{upper bound F}. In addition,   Lemma~\ref{upper bound F} also implies 
\begin{align}\tensor{K}h_T(\cdot)= e^{-\theta (T-\cdot)}\mathbbm{1}_{[0,T]}(\cdot)\int_{0}^T e^{-\theta (T-s)}s^{\beta-1}\dif s \le C e^{-\theta (T-\cdot)}\mathbbm{1}_{[0,T]}(\cdot),\label{Kht bound}
\end{align}  
 and
\begin{align*}
 \norm{\tensor{K}h_T}_{\FH_1}^2&\le C\norm{e^{-\theta (T-\cdot)} \mathbbm{1}_{[0,T]}(\cdot)}_{\FH_1 }^2 \le C C_\beta \Gamma(2\beta-1)\theta^{-2\beta},
\end{align*} 
where the last step is from the proof of Proposition \ref{prop bt lim}. 
By the inequality (\ref{norm.ineq}), we have the desired (\ref{ht to 011}).

Next, the inequality (\ref{inner product 00.ineq})  implies that 
\begin{align}\label{inner product.ineq}
\abs{\innp{  f_T,\,  h_T}_{\mathfrak{H}^{\otimes 2}}}\le \abs{\innp{  f_T,\,  h_T}_{\mathfrak{H}_1^{\otimes 2}} }+ \abs{\innp{  f_T,\,  h_T}_{\mathfrak{H}_2^{\otimes 2}}}+2C_{\beta}'\abs{\innp{\tensor{K} f_T,\, \tensor{K} h_T}_{\mathfrak{H}_1}}.
\end{align}
Denote $\vec{t}=(t_1,t_2),\,\, \vec{s}=(s_1,s_2)$. The symmetry of the functions $f_T$ and $h_T$ implies that 
\begin{eqnarray*}
\innp{f_T,\,  h_T}_{\mathfrak{H}_1^{\otimes 2}} & = & \int_{[0,T]^4}\,e^{-\left(T-t_1+ T-s_1+\abs{t_2-s_2} \right)\theta} \abs{t_1-t_2}^{2\beta-2}\abs{s_1-s_2}^{2\beta-2}\dif \vec{t}\dif \vec{s} \\
 & = & 2\int_{0\le s_1,t_2,s_2\le t_1\le T}\,e^{-\left(T-t_1+ T-s_1+\abs{t_2-s_2} \right)\theta} \abs{t_1-t_2}^{2\beta-2}\abs{s_1-s_2}^{2\beta-2}\dif \vec{t}\dif \vec{s}\\
 & & + \ 2\int_{0\le s_1,t_1,s_2\le t_2\le T}\,e^{-\left(T-t_1+ T-s_1+\abs{t_2-s_2} \right)\theta} \abs{t_1-t_2}^{2\beta-2}\abs{s_1-s_2}^{2\beta-2}\dif \vec{t}\dif \vec{s}.
\end{eqnarray*}
The L'H\^{o}pital's rule implies that 
\begin{align*}
&\lim_{T\to \infty}\int_{0\le s_1,t_2,s_2\le t_1\le T}\,e^{-\left(T-t_1+ T-s_1+\abs{t_2-s_2} \right)\theta} \abs{t_1-t_2}^{2\beta-2}\abs{s_1-s_2}^{2\beta-2}\dif \vec{t}\dif \vec{s}\\
&=\lim_{T\to \infty}\frac{1}{2\theta  } \int_{[0,T]^3} e^{-\left(  T-s_1+\abs{t_2-s_2} \right)\theta} (T-t_2 )^{2\beta-2}\abs{s_1-s_2}^{2\beta-2}\dif {t}_2\dif \vec{s}\\
&=\frac{4\beta-1}{4 \theta^{4\beta}}\Gamma^{2}(2\beta-1)\left( 1+ \frac{\Gamma(4\beta-1)\Gamma(3-4\beta)}{\Gamma(2-2\beta)\Gamma(2\beta)}\right),
\end{align*}where the last limit is from \cite{hu Nua 10}. The L'H\^{o}pital's rule implies that 
\begin{align*}
& \lim_{T\to \infty} \int_{0\le s_1,t_1,s_2\le t_2\le T}\,e^{-\left(T-t_1+ T-s_1+\abs{t_2-s_2} \right)\theta} \abs{t_1-t_2}^{2\beta-2}\abs{s_1-s_2}^{2\beta-2}\dif \vec{t}\dif \vec{s}\\
&=\lim_{T\to \infty}\frac{1}{2\theta  } \int_{[0,T]^3} e^{-\left( T-t_1+ T-s_1+T-s_2 \right)\theta} (T-t_1 )^{2\beta-2}\abs{s_1-s_2}^{2\beta-2}\dif {t}_1\dif \vec{s}\\
& = C \lim_{T\to \infty} \int_{[0,T]^2} e^{-\left(  T-s_1+T-s_2 \right)\theta}  \abs{s_1-s_2}^{2\beta-2} \dif \vec{s}\\
& = C { \lim_{T \to \infty} \|e^{-\theta(T-\cdot)} \mathbbm{1}_{[0,T]} (\cdot) \|_{\mathfrak{H}_1} \leq C } \,,
\end{align*}
{where we have used \eqref{ht h1} in the last step.} Hence we have $\abs{\innp{  f_T,\,  h_T}_{\mathfrak{H}_1^{\otimes 2}}} \leq C$. \\
The symmetry and Lemma~\ref{upper bound F} imply that 
\begin{align*}
\innp{  f_T,\,  h_T}_{\mathfrak{H}_2^{\otimes 2}}&=\int_{[0,T]^4}\,e^{-\left(T-t_1+ T-s_1+\abs{t_2-s_2} \right)\theta} (t_1 t_2 s_1s_2)^{ \beta-1}\dif \vec{t}\dif \vec{s}\\
&= \left(\int_0^T e^{-(T-t_1)\theta }t_1^{\beta-1}\dif t_1\right)^2\int_{[0,T]^2} e^{-  \theta\abs{t_2-s_2} } (  t_2  s_2)^{ \beta-1}\dif {t}_2\dif  {s}_2\\
&\le C T^{2(\beta-1)}\int_{[0,T]^2} e^{-  \theta \abs{t_2-s_2} } (  t_2  s_2)^{ \beta-1}\dif {t}_2\dif  {s}_2\\
&\le  C T^{4\beta-3},
\end{align*} where the last line is from the inequality (\ref{upper bound 0011}).\\
The inequalities (\ref{Kft bound}) and (\ref{Kht bound}) imply that 
\begin{align*}
\innp{\tensor{K} f_T,\, \tensor{K} h_T}_{\mathfrak{H}_1}&\le C \int_{[0,T]^2} u^{\beta-1} e^{-\theta (T-v)} \abs{u-v}^{2\beta-2} \dif u\dif v.
\end{align*}
We consider the right hand side on two subregions of $[0,T]^2$. On  $\set{0\le u\le v \le T }$, we make the change of variable $a=v-u$ and apply the proof of Lemma~\ref{upper bound F} to obtain
\begin{align*}
\int_{0\le u\le v \le T} u^{\beta-1} e^{-\theta (T-v)} \abs{u-v}^{2\beta-2} \dif u\dif v&=\int_{0\le a\le v \le T} (v-a)^{\beta-1} e^{-\theta (T-v)} a^{2\beta-2} \dif a\dif v\\
&=\int_0^1 (1-x)^{\beta-1} x^{2\beta-2} dx \cdot \int_0^T e^{-\theta (T-v)} v^{3\beta-2} \dif v \\
&\le C T^{3\beta-2}.
\end{align*} On  $\set{0\le v\le u \le T }$, we make the change of variable $a=u-v$ and apply Lemma~\ref{upper bound F} to obtain
\begin{align*}
\int_{0\le v\le u \le T} u^{\beta-1} e^{-\theta (T-v)} \abs{u-v}^{2\beta-2} \dif u\dif v&=\int_{0\le a\le u \le T} u^{\beta-1} e^{-\theta (T-(u-a))} a^{2\beta-2} \dif a\dif u\\
&\le \frac{\Gamma(2\beta-1)}{\theta^{2\beta-1}}\int_0^T u^{\beta-1} e^{-\theta (T-u)} \dif u \\
&\le C.
\end{align*}  Hence, we have $\innp{\tensor{K} f_T,\, \tensor{K} h_T}_{\mathfrak{H}_1}\le C (1+ T^{3\beta-2})$. Substituting the upper bounds of $$ {\innp{  f_T,\,  h_T}_{\mathfrak{H}_1^{\otimes 2}}},\, \innp{  f_T,\,  h_T}_{\mathfrak{H}_2^{\otimes 2}},\, \innp{\tensor{K} f_T,\, \tensor{K} h_T}_{\mathfrak{H}_1}$$ into the inequality (\ref{inner product.ineq}), we have the desired (\ref{ht ft to 011}).  
\end{proof}
 \begin{proposition}\label{cor b-e 2}
 Denote $a:=C_{\beta} \Gamma(2\beta-1) \theta^{-2\beta}$ and $$Q_T := \frac{F_T-H_T }{ \sqrt{T}} $$ where $F_T$ and $H_T$ are given in \eqref{Ft} and \eqref{ht} respectively. The constant $\gamma $ that depends on $\beta$ is defined in Theorem~\ref{B-E bound thm}. Then there exists a constant $C_{\theta, \beta}$ such that when $T$ is large enough,
 \begin{align}
 \sup_{z\in \Rnum}\abs{ P(\frac{1}{\sqrt{ 4\theta a^2\sigma_{\beta}^2}} Q_T\le z)- P(Z\le z)}\le \frac{ C_{\theta, \beta}}{{T^{\gamma}}}.
 \end{align}
 \end{proposition}
 \begin{proof} It follows from (\ref{asy norm moment})-(\ref{asy norm moment 2}) that 
 \begin{align*}
 \frac{1}{\sqrt{ 4\theta a^2\sigma_{\beta}^2}}\, Q_T \stackrel{ {law}}{\to}  \mathcal{N}(0,1).
 \end{align*}
 By the Fourth moment Berry-Ess\'{e}en bound (see, for example, Corollary 5.2.10 of \cite{Nou 12}), we have
 \begin{align*}
&\quad \sup_{z\in \Rnum}\abs{ P(\frac{1}{ \sqrt {4\theta a^2\sigma_{\beta}^2}} \, Q_T\le z)- P(Z\le z)}\\
 &\le\sqrt{\frac{\mathbb{E}[Q_T^4]-3 \mathbb{E}[Q_T^2]^2}{3  \mathbb{E}[Q_T^2]^2} } + \frac{\abs{\mathbb{E}[Q_T^2]- {4}\theta a^2\sigma_{\beta}^2 }}{\mathbb{E}[Q_T^2] \vee ( {4}\theta a^2\sigma_{\beta}^2)}
 \end{align*}
Lemma \ref{2ft2} and Lemma \ref{htto0} imply that the second term is bounded by $\frac{C}{T^{{\gamma}}}$. For the first term, {we have $\mathbb{E}[Q_T^2]\to 4\theta a^2\sigma_{\beta}^2  $, so we only need to show} when $T$ is large enough,
 \begin{align}\label{ Fourth moment Berry-Esseen bound}
 \E[Q_T^4]-3 \E[Q_T^2]^2\le \frac{C}{T^{ 2\gamma  }}.
 \end{align}
 In fact,  we have
 \begin{align}
 &  \E[Q_T^4]-3 \big(\E[Q_T^2]\big)^2 \nonumber\\
 &=\E\left[\left(\frac{F_T }{   \sqrt{T}}\right)^4\right] -3 \left[\E \left(\frac{F_T }{   \sqrt{T}}\right)^2\right]^2  - 3 \left\{\left[\E[Q_T^2]\right]^2 - \left[\E \left(\frac{F_T }{   \sqrt{T}}\right)^2\right]^2 \right\} \nonumber \\
 &+ \E\left[\left(\frac{H_T }{   \sqrt{T}}\right)^4\right]  +6 \E \left[\left(\frac{F_T }{   \sqrt{T}}\right)^2 \left(\frac{H_T }{   \sqrt{T}}\right)^2 \right]  -4\E \left[\left(\frac{F_T }{   \sqrt{T}}\right)^3 \left(\frac{H_T }{   \sqrt{T}}\right) \right] -4\E \left[\left(\frac{F_T }{   \sqrt{T}}\right) \left(\frac{H_T }{   \sqrt{T}}\right)^3 \right]  . \label{qt decomp}
 \end{align} 
 Proposition~\ref{contraction ft} and Eq. (5.2.5) of \cite{Nou 12} imply that
 \begin{align*}
 \E\left[\left(\frac{F_T }{  \sqrt{T}}\right)^4\right] -3 \left[\E \left(\frac{F_T }{   \sqrt{T}}\right)^2\right]^2 \le C \left(\frac{1}{T}\|f_T \otimes_1 f_T\|_{\mathfrak{H}^{\otimes 2}}\right)^2\le \frac{C}{T^{2\gamma}} \,.
 \end{align*}
 Lemma \ref{2ft2}, Lemma \ref{htto0} and the Cauchy-Schwarz inequality imply that   
 \begin{align*}
\abs{ \left[\E[Q_T^2]\right]^2 - \left[\E \left(\frac{F_T }{   \sqrt{T}}\right)^2\right]^2 }&\le  \E\left[ Q_T^2 +\left(\frac{F_T }{   \sqrt{T}}\right)^2\right]   \E \abs{  \frac{H_T }{   \sqrt{T}} \left(\frac{H_T-2 F_T }{\sqrt{T}} \right)   }  \\
& \le\frac{C}{ {T}} (1+T^{3\beta-2})\le \frac{C}{T^{2\gamma}}.
 \end{align*} From the Hypercontractivity of multiple Wiener-It\^{o} integrals, the Cauchy-Schwarz inequality, Lemma \ref{2ft2} and Lemma \ref{htto0},
the other three terms containing squared, cubic, and quartic $\frac{H_T }{   \sqrt{T}}  $ in \eqref{qt decomp} are all bounded by $\frac{C}{ {T}}$.\\
Finally, we deal with the term $\E \left[\left(\frac{F_T }{   \sqrt{T}}\right)^3 \left(\frac{H_T }{   \sqrt{T}}\right) \right]  $.  
{We apply the product formula of \eqref{ito.prod} for the term $F_T^3$ and use the orthogonality of multiple Wiener-It\^{o} integrals to obtain}
\begin{align}
\E \left[  {F_T }^3  {H_T } \right]  &=  16 \E[I_2((f_T\tilde{\otimes}_1f_T )\tilde{\otimes}_1 f_T) H_T ]+ 12 \E[I_2((f_T\tilde{\otimes}f_T )\tilde{\otimes}_2 f_T) H_T ] \nonumber\\
&+2\norm{f_T}^2_{\FH^{\otimes 2}} \E[F_T H_T]. \label{ft3 ht}
\end{align}
{We will deal with the three items on the right-hand side of the above equation} \eqref{ft3 ht}. First, we apply It\^{o}'s isometry and the Cauchy-Schwarz inequality to obtain
\begin{align*}
\abs{\E[I_2((f_T\tilde{\otimes}_1f_T )\tilde{\otimes}_1 f_T) H_T ]} &=2\abs{\innp{(f_T\tilde{\otimes}_1f_T )\tilde{\otimes}_1 f_T,\,h_T }_{\FH^{\otimes 2}}}\\
&\le \norm{f_T\tilde{\otimes}_1f_T }_{\FH^{\otimes 2}} \cdot \norm{f_T}_{\FH^{\otimes 2}} \cdot \norm{h_T}_{\FH^{\otimes 2}}\\
&\le C T^{\frac32-\gamma}.
\end{align*} where the last inequality is from Proposition~\ref{contraction ft}, Lemma \ref{2ft2} and the inequality (\ref{ht to 011}).\\
Second, we apply Lemma \ref{2ft2} and the inequality (\ref{ht ft to 011}) to obtain
\begin{align}\label{ft norm ft ht}
 \norm{f_T}^2_{\FH^{\otimes 2}} \E[F_T H_T] = \norm{f_T}^2_{\FH^{\otimes 2}}\cdot \abs{\innp{f_T ,\, h_T}_{\FH^{\otimes 2}}}\le C (T+T^{3\beta-1}).
\end{align}
Third, the symmetry of $f_T$ implies that  on $[0,T]^4$,
\begin{align}
f_T\tilde{\otimes}f_T (u,v,x,y)&=\frac13 \left[ e^{-\theta (\abs{u-v}+\abs{x-y}  )} +  e^{-\theta (\abs{u-x}+\abs{v-y}  )}+ e^{-\theta (\abs{u-y}+\abs{v-x}  )}\right], \nonumber\\
(f_T\tilde{\otimes}f_T )\tilde{\otimes}_2 f_T (u,v)&=\frac13  \left[ \norm{f_T}^2_{\FH^{\otimes 2}}\cdot f_T(u,v) + 2 \innp{e^{-\theta (\abs{u-x}+\abs{v-y}  )} ,\,  f_T(x,y)}_{\FH^{\otimes 2}}\right]. \label{fttensors}
\end{align}
We apply Fubini's theorem, the Cauchy-Schwarz inequality and Proposition~\ref{contraction ft} to obtain
\begin{align}
\abs{ \left\langle \innp{e^{-\theta (\abs{u-x}+\abs{v-y}  )} ,\,  f_T(x,y)}_{\FH^{\otimes 2}},\,h_T(u,v)\right\rangle_{\FH^{\otimes 2}}}
&=\abs{\innp{(f_T {\otimes}_1f_T ) {\otimes}_1 f_T,\,h_T }_{\FH^{\otimes 2}}}\le C T^{\frac32-\gamma}. \label{exp ft ht}
\end{align}
Hence, by \eqref{fttensors}, \eqref{exp ft ht}, \eqref{ft norm ft ht}, and It\^{o} isometry,
\begin{align*}
\E[I_2((f_T\tilde{\otimes}f_T )\tilde{\otimes}_2 f_T) H_T ]\le C (T^{\frac32-\gamma}+ T+T^{3\beta-1}).
\end{align*}
{By the above arguments, we obtain}
\begin{align*}
\abs{\E \left[\left(\frac{F_T }{   \sqrt{T}}\right)^3 \left(\frac{H_T }{   \sqrt{T}}\right) \right] }\le C (T^{-\frac12-\gamma}+ T^{-1}+T^{3(\beta-1)})\le \frac{C}{T^{2\gamma}}.
\end{align*}
Therefore, the desired inequality (\ref{ Fourth moment Berry-Esseen bound}) holds. {This concludes the proof.} 
\end{proof}
 
\noindent{\it Proof of Theorem~\ref{B-E bound thm}.\,} 
It follows from the equations (\ref{bt bt})-(\ref{ratio 1}), Theorem~\ref{kp}, and Proposition~\ref{prop bt lim}  that  there exists a constant $C $ independent of $T$ such that for $T$ large enough,
\begin{align}
    &\sup_{z\in \Rnum}\abs{P\Big(\sqrt{\frac{T}{\theta \sigma^2_{\beta}} }(\hat{\theta}_T-\theta )\le z\Big)-P(Z\le z)} \nonumber \\
& \le C \times \max \Big(\abs{b_T^2-\frac{\norm{f_T}_{\mathfrak{H}^{\otimes 2}}^2}{2\theta \sigma^2_{\beta}T}},\,\frac{1}{T}\norm{f_T\otimes_1 f_T}_{\mathfrak{H}^{\otimes 2}},\, \frac{1}{\sqrt{T}}\norm{f_T\otimes_1 g_T}_{\mathfrak{H}^{\otimes 2}},\nonumber \\
& \frac{1}{\sqrt{T}}\innp{f_T,\,g_T}_{\mathfrak{H}^{\otimes 2}},\,\norm{g_T}_{\mathfrak{H}^{\otimes 2}}^2,\,\norm{g_T\otimes_1 g_T}_{\mathfrak{H}^{\otimes 2}}\Big),\label{sup le}
\end{align}
where $g_T$ is given by \eqref{gt ts}. Denote $a= C_{\beta} \Gamma(2\beta-1) \theta^{-2\beta} $. Then Remark~\ref{rem 36} and Lemma~\ref{2ft2} imply that  there exists a constant $C>0$ such that  for $T$ large enough,
\begin{align*}
\abs{b_T^2-\frac{\norm{f_T}^2_{\mathfrak{H}^{\otimes 2}}}{2\theta \sigma^2_{\beta}T}}\le \abs{b_T^2- a^2}+\abs{\frac{\norm{f_T}^2}{2\theta \sigma^2_{\beta}T}-a^2}\le   \frac{C}{T^{3-4\beta}}.
\end{align*}
Since $g_T=({f_T-h_T} )/{2\theta T}$, we apply Minkowski's inequality, Lemma~\ref{2ft2} and Lemma \ref{htto0} to conclude that there exists a constant $C>0$ such that  for $T$ large enough,
\begin{align*}
\norm{g_T}_{\FH^{\otimes 2}}\le \frac{C}{\sqrt{T}} \left[\norm{\frac{f_T}{\sqrt{T}}}_{\FH^{\otimes 2}} + \norm{\frac{h_T}{\sqrt{T}}}_{\FH^{\otimes 2}}\right]\le  \frac{C}{\sqrt{T}},
\end{align*} which, together with the Cauchy-Schwarz inequality and  Lemmas~\ref{2ft2}, implies 
\begin{align*}
 \frac{1}{\sqrt{T}}\norm{f_T\otimes_1 g_T}_{\FH^{\otimes 2}} \le   \frac{C}{\sqrt{T}},\ \frac{1}{\sqrt{T}}\abs{\innp{f_T,\,g_T}_{\FH^{\otimes 2}}} \le   \frac{C}{\sqrt{T}}, \ \norm{g_T\otimes_1 g_T}_{\FH^{\otimes 2}}\le   \frac{C}{{T}} ,\ \norm{g_T}^2_{\FH^{\otimes 2}}\le   \frac{C}{{T}}.
\end{align*}
Substituting (\ref{zhou ineq}) and the above inequalities into (\ref{sup le}), we obtain the desired Berry-Ess\'{e}en bound (\ref{b-e bound 34}).

The Berry-Ess\'{e}en bound (\ref{b-e bound 44}) can be obtained by the similar arguments of Theorem 3.2 in \cite{SV 18}. 
Denote
   $$A :=P(\sqrt{\frac{4\beta^2 T}{\theta \sigma^2_{\beta}}} (\tilde{\theta}_T-\theta )\le z) - P(Z\le z) \,.$$
Since $\tilde{\theta}_T>0$, we shall suppose $z>-\sqrt{\frac{4\beta^2 T}{\theta \sigma^2_{\beta}}}\theta $. Otherwise, the standard estimate for a normal random variable $P(\abs{Z}\ge t)\le \frac{1}{t},\,\forall t>0$ yields $$|A|=P(Z\le z) \leq \frac{C}{\sqrt{T}} \,.$$
Now by (\ref{theta tilde formula}) for the formula of $\tilde{\theta}_T$ , we have
\begin{align*}
A &=P\left( \tilde{\theta}_T-\theta \le \sqrt{\frac{\theta \sigma^2_{\beta}}{4\beta^2 T}} z\right) - P(Z\le z)\\
&=P\left( \frac{1}{C_{\beta}  \Gamma(2\beta-1) T} \int_0^T X_t^2\mathrm{d} t  \ge  \big( \sqrt{\frac{\theta \sigma^2_{\beta}}{4\beta^2 T}}z+\theta\big)^{-2\beta} \right) - P(Z\le z)\\
&= P\left( \frac{1}{T} \int_0^T X_t^2\mathrm{d} t -a \ge \Gamma(2\beta-1)C_{\beta}\Big[ \big( \sqrt{\frac{\theta \sigma^2_{\beta}}{4\beta^2 T}}z+\theta\big)^{-2\beta}-\theta^{-2\beta}\Big]\right) - P(Z\le z)\\
&=P\left( \frac{Q_T}{ \sqrt{T}}+ 2\theta(b_T-a) \ge 2 a  \theta \Big[ \big( 1+\frac{z \sigma_{\beta}}{2\beta \sqrt{\theta T}}\big)^{-2\beta}-1 \Big] \right) - P(Z\le z),
\end{align*}
where {in the last step we have used \eqref{asy norm moment 2} and the term $Q_T = \frac{F_T - H_T}{\sqrt{T}}$ as given in Proposition~\ref{cor b-e 2}.}
We take the short-hand notation $ \bar{\Phi}(z)=1-P(Z\le z)$ and 
$$\nu=\sqrt{  \frac{\theta T}{   \sigma_{\beta}^2}}   \Big[ \big( 1+\frac{z \sigma_{\beta}}{2\beta \sqrt{\theta T}}\big)^{-2\beta}-1 \Big].$$ 
Then 
\begin{align*}
|A| 
&= \abs{ P\left( \frac{1}{\sqrt{ 4\theta a^2\sigma_{\beta}^2}}\big(  {Q_T}+ 2\theta { \sqrt{T}} (b_T-a)\big) \ge   \nu\right)- P(Z\le z) }\\
&\le   \abs{ P\left( \frac{ {Q_T}}{\sqrt{ 4\theta a^2\sigma_{\beta}^2}}  \ge   \nu -  \frac{  {  \sqrt{\theta T}} (b_T-a)}{\sqrt{   a^2\sigma_{\beta}^2}} \right)-\bar{\Phi} \left(\nu - \frac{  { \sqrt{\theta T}} (b_T-a)}{\sqrt{   a^2\sigma_{\beta}^2}}\right)} \\
&+ \abs{\bar{\Phi} \left(\nu- \frac{ { \sqrt{\theta T}} (b_T-a)}{\sqrt{   a^2\sigma_{\beta}^2}} \right)-\bar{\Phi}(\nu) } + \abs{\bar{\Phi}(\nu)-P(Z\le z) } ,
\end{align*} 
where the first term is bounded by $\frac{C}{T^{\gamma}}$ from Proposition~\ref{cor b-e 2}, the third term is bounded by $\frac{C}{\sqrt{T} }$ from Lemma~\ref{bound lem 54} below,  and the second term is bound by $\frac{C}{T^{\frac{3-4\beta}{2}}}$ from Remark~\ref{rem 36}  and the standard estimate for the tail of a normal random variable, $\abs{\bar{\Phi}(z_1)- \bar{\Phi}(z_2)}\le \abs{z_1-z_2} $. 
{\hfill\large{$\Box$}}\\

\begin{lemma}\label{bound lem 54}
 Let $c>0$ be a constant. Denote $\nu(z)=  \frac{c}{ 2\beta}  \sqrt{T}  \Big[ \big( 1+\frac{z }{c \sqrt{T}}\big)^{-2\beta}-1 \Big]$ when $z> -c \sqrt{T}$ and $ \bar{\Phi}(z)=1-P(Z\le z)$. Then there exists some positive number $C$ independent of $T$ such that 
\begin{align*}
 \sup_{z> -c \sqrt{T}}\abs{  \bar{\Phi}(\nu) -\Phi (z)}\le \frac{C}{\sqrt{T}}.
\end{align*}
\end{lemma}
\begin{proof} We follow the line of the proof of Theorem 3.2  in \cite{SV 18}.
By the mean value theorem, there exists some number $\eta\in (0,1)$ such that 
\begin{align*}
\nu=- z \big( 1+\frac{z \eta}{c \sqrt{T}}\big)^{-2\beta-1} .
\end{align*}Hence, 
\begin{align*}
\abs{  \bar{\Phi}(\nu) -\Phi (z)}&=\abs{  {\Phi}\Big( \big( 1+\frac{z \eta}{c \sqrt{T}}\big)^{-2\beta-1}\cdot z\Big) -\Phi (z)}\\
&=\frac{1}{\sqrt{2\pi}}  {\int^{z}_{ z \big( 1+\frac{z \eta}{c \sqrt{T}}\big)^{-2\beta-1} } \,e^{-\frac{t^2}{2}}\dif t} \,.
\end{align*}
When $z\in (-c \sqrt{T}, -\frac12 c \sqrt{T}]$, it is obvious that
\begin{align*}
\frac{1}{\sqrt{2\pi}}  {\int_{ z \big( 1+\frac{z \eta}{c \sqrt{T}}\big)^{-2\beta-1}  }^{z} \,e^{-\frac{t^2}{2}}\dif t}\le \Phi(-\frac12 c \sqrt{T})\le \frac{C}{\sqrt{T}}.
\end{align*}
When $z\in( -\frac12 c \sqrt{T},\,0)$, we have
\begin{align}\label{eqn. lem 54}
\frac{1}{\sqrt{2\pi}}  {\int_{z \big( 1+\frac{z \eta}{c \sqrt{T}}\big)^{-2\beta-1}  }^{z} \,e^{-\frac{t^2}{2}}\dif t}&\le \abs{z} e^{-\frac{z^2}{2}}\Big(  \big( 1+\frac{z \eta}{c \sqrt{T}}\big)^{-2\beta-1} -1\Big) .
\end{align}The mean value theorem implies that there exists some number $\eta'\in (0,1)$ such that
\begin{align*}
 \big( 1+\frac{z \eta}{c \sqrt{T}}\big)^{-2\beta-1} -1= (-1-2\beta) \frac{z\eta}{c\sqrt{T}} \big( 1+\frac{z \eta\eta'}{c \sqrt{T}}\big)^{-2\beta-2}\le (1+2\beta) \left(\frac12\right)^{-2\beta-2}\frac{\abs{z}}{c\sqrt{T}}.
\end{align*} 
Substituting the above inequality into (\ref{eqn. lem 54}), and 
since the function $f(z)= {z}^2 e^{-\frac{z^2}{2}}$ is uniformly bounded, we have that when $z\in( -\frac12 c \sqrt{T},\,0)$, 
\begin{align*}
\frac{1}{\sqrt{2\pi}}  {\int_{z \big( 1+\frac{z \eta}{c \sqrt{T}}\big)^{-2\beta-1}  }^{z} \,e^{-\frac{t^2}{2}}\dif t}&\le \frac{C}{\sqrt{T}}.
\end{align*}
When $z\ge 0$, using the mean value theorem and making the change of variable $t = z^2 s$ together with the fact that $f_2(s, z) = z^2e^{-\frac{s^2 z^4}{2}}$ is also uniformly
bounded, we conclude that there exists a number $\eta'\in (0,1)$ such that 
\begin{align*}
 {\int_{ z \big( 1+\frac{z \eta}{c \sqrt{T}}\big)^{-2\beta-1}  }^{z} \,e^{-\frac{t^2}{2}}\dif t} &={\int_{ \frac{1}{z} \big( 1+\frac{z \eta}{c \sqrt{T}}\big)^{-2\beta-1}  }^{\frac{1}{z}} \,\,z^2e^{-\frac{s^2z^4}{2}}\dif s}\\
 &\le C \frac{1}{z} \Big( 1-  \big( 1+\frac{z \eta}{c \sqrt{T}}\big)^{-2\beta-1}  \Big)\\
 &= C  (1+2\beta) \frac{1}{z} \big( 1+\frac{z \eta\eta'}{c \sqrt{T}}\big)^{-2\beta-2}  \frac{z \eta}{c \sqrt{T}} \\
 &\le \frac{C}{\sqrt{T}}.
\end{align*}
\end{proof}
\vskip 0.2cm {\small {\bf  Acknowledgements}:
 Y. Chen is supported by NSFC (No.11871079, No.11961033, No.11961034).
}



\end{document}